\documentclass[final]{siamltex}
\usepackage[T1]{fontenc}
\usepackage{lmodern}
\usepackage{amsmath}
\usepackage{amsfonts}
\usepackage{amssymb}
\usepackage{xspace}
\usepackage{graphicx}
\usepackage{subfigure}
\usepackage[notref]{showkeys}
\usepackage{color}
\usepackage{ulem}
\usepackage[pdftex,colorlinks=true]{hyperref}
\usepackage{thumbpdf}

\newcommand{\ipdots}[1][]{\ensuremath{{\left\langle{\cdot},{\cdot}\right\rangle}_{#1}}}

\newcommand{\ASS}{\ensuremath\mathrel{\mathop:}=}
\newcommand{\SSA}{=\ensuremath\mathrel{\mathop:}}

\newcommand{\DEF}{\ASS}

\newcommand{\Epub}[1]{} 
\newcommand{\ETAL}{{\sl et al.}\@\xspace}

\newcommand{\FED}{\SSA}

\newcommand{\LLRA}{\Longleftrightarrow}

\newcommand{\q}{\quad}
\newcommand{\qq}{\quad\quad}

\newcommand{\Rec}[1]{}   

\newcommand{\WH}{\widehat}

\newcommand{\ColVec}[1]{\left[\begin{array}{c}#1\end{array}\right]}

\newcommand{\Matrix}[2]{\left[\begin{array}{#1}#2\end{array}\right]}

\newcommand{\CC}{{\mathbb C}}

\newcommand{\NN}{{\mathbb N}}
\newcommand{\RR}{{\mathbb R}}

\newcommand{\A}{\star}

\renewcommand{\H}{\mathsf{H}}

\newcommand{\Image}{\mathsf{im}\,}

\newcommand{\Ker}{\mathsf{ker}\,}

\newcommand{\Span}{\mathsf{span}\,}

\newcommand{\T}{\mathsf{T}}

\newcommand{\bfA}{{\mathbf A}}
\newcommand{\bfAA}{{\mathbf A}\kern-.2em^\A}
\newcommand{\bfAH}{{\mathbf A}\kern-.2em^\H}
\newcommand{\bfAmH}{{\mathbf A}\kern-.2em^{-\H}}
\newcommand{\bfAT}{{\mathbf A}\kern-.15em^\T}
\newcommand{\bfAhat}{\widehat{\mathbf A}}

\newcommand{\bfB}{{\mathbf B}}
\newcommand{\bfC}{{\mathbf C}}
\newcommand{\bfD}{{\mathbf D}}

\newcommand{\bfE}{{\mathbf E}}
\newcommand{\bfEB}{{\mathbf E}_\bfB}

\newcommand{\bfI}{{\mathbf I}}
\newcommand{\bfJ}{{\mathbf J}}

\newcommand{\bfMA}{{\mathbf M}_{\mathbf A}}
\newcommand{\bfMB}{{\mathbf M}_{\mathbf B}}
\newcommand{\bfMI}{{\mathbf M}_{\mathbf I}}

\newcommand{\bfP}{{\mathbf P}}
\newcommand{\bfPB}{{\mathbf P}_\bfB}
\newcommand{\bfPA}{{\mathbf P}_\bfA}
\newcommand{\bfPI}{{\mathbf P}_\bfI}

\newcommand{\bfQ}{{\mathbf Q}}
\newcommand{\bfQB}{{\mathbf Q}_\bfB}
\newcommand{\bfQA}{{\mathbf Q}_\bfA}
\newcommand{\bfQI}{{\mathbf Q}_\bfI}
\newcommand{\bfR}{{\mathbf R}}

\newcommand{\bfS}{{\mathbf S}}

\newcommand{\bfU}{{\mathbf U}}
\newcommand{\bfV}{{\mathbf V}}

\newcommand{\bfW}{{\mathbf W}}

\newcommand{\bfY}{{\mathbf Y}}

\newcommand{\bfZ}{{\mathbf Z}}

\newcommand{\bfb}{{\mathbf b}}
\newcommand{\bfbhat}{\widehat{\mathbf b}}

\newcommand{\bff}{{\mathbf f}}

\newcommand{\bfh}{{\mathbf h}}

\newcommand{\bfr}{{\mathbf r}}
\newcommand{\bfrb}{\overline{\mathbf r}}
\newcommand{\bfrhat}{\widehat{\mathbf r}}

\newcommand{\bfs}{{\mathbf s}}

\newcommand{\bfu}{{\mathbf u}}

\newcommand{\bfv}{{\mathbf v}}
\newcommand{\bfvhat}{\widehat{\mathbf v}}

\newcommand{\bfw}{{\mathbf w}}

\newcommand{\bfx}{{\mathbf x}}
\newcommand{\bfxb}{\overline{\mathbf x}}

\newcommand{\bfxhat}{\widehat{\mathbf x}}

\newcommand{\bfy}{{\mathbf y}}

\newcommand{\bfyhat}{\widehat{\mathbf y}}

\newcommand{\calK}{\mathcal{K}}

\newcommand{\calS}{\mathcal{S}}
\newcommand{\calU}{\mathcal{U}}

\newcommand{\calX}{\mathcal{X}}

\newcommand{\indMLBICG}[1]{^{\mbox{\tiny\sf ML$(k)$BiCG}}}
\newcommand{\indMLSTAB}[1]{^{\mbox{\tiny\sf ML$(k)$BiCGSTAB}}}

\newcommand{\BICG}{{\sc BiCG}}

\newcommand{\BICGSTAB}{{\sc BiCG\-Stab}}

\newcommand{\CG}{{\sc CG}}
\newcommand{\CGNE}{{\sc CGNE}}
\newcommand{\CGNR}{{\sc CGNR}}
\newcommand{\CGS}{{\sc CGS}}

\newcommand{\CR}{{\sc CR}}

\newcommand{\FOM}{{\sc FOM}}
\newcommand{\GCG}{{\sc GCG}}
\newcommand{\GCR}{{\sc GCR}}
\newcommand{\GCRO}{{\sc GCRO}}
\newcommand{\GCROT}{{\sc GCROT}}
\newcommand{\GMERR}{{\sc GMErr}}
\newcommand{\GMRES}{{\sc GMRes}}

\newcommand{\IDRS}{{\sc IDR(s)}}

\newcommand{\MINRES}{{\sc MinRes}}

\newcommand{\RMINRES}{{\sc RMinRes}}

\newcommand{\SYMMLQ}{{\sc SymmLQ}}

\def\addots{\mathinner{\mskip1mu\raise1pt\hbox{.}\mskip2mu\raise4pt\hbox{.}
   \mskip2mu\raise7pt\vbox{\kern7pt\hbox{.}}\mskip1mu}}

\newcommand{\NOTE}[1]{}                  
\newcommand{\QUEST}[1]{}                

\renewcommand{\em}{\it}

\newtheorem{example}[theorem]{Example}

\hypersetup{
    pdftitle={A framework for deflated and augmented Krylov subspace methods},
	pdfauthor={Andr\'{e} Gaul, Martin H. Gutknecht, J\"{o}rg Liesen, Reinhard Nabben},
	pdfkeywords={Krylov subspace methods, augmentation, deflation, subspace recycling, CG, MINRES, GMRES}
}

\title{A framework for deflated and augmented Krylov subspace
methods\footnotemark[1]}
\author{
	Andr\'{e} Gaul\footnotemark[3],
	Martin H. Gutknecht\footnotemark[2],
	J\"org Liesen\footnotemark[3] and
	Reinhard Nabben\footnotemark[3]
}
\begin{document}
	\renewcommand{\thefootnote}{\fnsymbol{footnote}}
    \footnotetext[1]{Version of \today}
	\footnotetext[2]{
		Seminar for Applied Mathematics, ETH Zurich, CH-8092 Zurich, Switzerland
		({\tt mhg@math.ethz.ch}).  Work started while this author was visiting the
		TU Berlin, supported by the DFG Forschungszentrum MATHEON and the Mercator
		Visiting Professorship Program of the DFG.
	}
	\footnotetext[3]{
		Institut f\"ur Mathematik, Technische Universit\"{a}t Berlin, Stra{\ss}e
		des 17. Juni 136, D-10623 Berlin, Germany ({\tt
		\{gaul,liesen,nabben\}@math.tu-berlin.de}). The work of Andr\'{e} Gaul,
		J\"{o}rg Liesen and Reinhard Nabben was supported by the DFG
		Forschungszentrum MATHEON. The work of J\"{o}rg Liesen was supported by the
		Heisenberg Program of the DFG.
	}
	\renewcommand{\thefootnote}{\arabic{footnote}}

	\maketitle

	\begin{abstract}
        We consider deflation and augmentation techniques for accelerating the
        convergence of Krylov subspace methods for the solution of nonsingular
        linear algebraic systems.  Despite some formal similarity, the two
        techniques are conceptually different from preconditioning.  Deflation
        (in the sense the term is used here) ``removes'' certain parts from the
        operator making it singular, while augmentation adds a subspace to the
        Krylov subspace (often the one that is generated by the singular
        operator); in contrast, preconditioning changes the spectrum of the
        operator without making it singular.  Deflation and augmentation have
        been used in a variety of methods and settings.  Typically, deflation is
        combined with augmentation to compensate for the singularity of the
        operator, but both techniques can be applied separately.

        We introduce a framework of Krylov subspace methods that satisfy a
        Galerkin condition. It includes the families of orthogonal residual (OR)
        and minimal residual (MR) methods.  We show that in this framework
        augmentation can be achieved either explicitly or, equivalently,
        implicitly by projecting the residuals appropriately and correcting the
        approximate solutions in a final step.
        We study conditions for a breakdown of the deflated methods, and we show
        several possibilities to avoid such breakdowns for the
        deflated \MINRES{} method. Numerical experiments illustrate properties
        of different variants of deflated \MINRES{} analyzed in this paper.
	\end{abstract}



	\section{Introduction}

    There are numerous techniques to accelerate the speed of convergence of
    Krylov subspace methods for solving large linear algebraic systems
	\begin{equation}
		\label{Axb}
		\bfA \bfx = \bfb,
	\end{equation}
    where $\bfA \in \CC^{N \times N}$ is nonsingular and $\bfb \in \CC^N$.  The
    most widely used technique is {\em preconditioning}. Here the system
    \eqref{Axb} is modified using left- and/or right-multiplications with a {\em
    nonsingular} matrix (called the preconditioner). A typical goal of
    preconditioning is to obtain a modified matrix that is in some sense close
    to the identity matrix.  For surveys of preconditioning techniques we refer
    to the books by Greenbaum~\cite[Part~II]{Gre97} and
    Saad~\cite[Chapters~9--14]{Saa03} and the survey of Benzi \cite{Ben02}.

    Here we consider two approaches for convergence acceleration that are called
    {\em deflation} and {\em augmentation}.  Let us briefly describe the main
    ideas of the two techniques.  In {\em deflation} the system \eqref{Axb} is
    multiplied (at least implicitly) with a suitably chosen projection, and the
    general goal is to ``eliminate'' components that supposedly slow down
    convergence.  Typically these are components that correspond to small
    eigenvalues.  Multiplication by the projection turns the system \eqref{Axb}
    into a consistent singular one, which is then solved by a Krylov subspace
    method.  We need to mention, however, that techniques have been proposed
    that move small eigenvalues of $\bfA$ to some large common value, say, to
    the value $1$; see \cite{BagCGR98,ErhBP96,KhaY95}.  Some authors refer to
    these techniques as ``deflation'' too.  In {\em augmentation} techniques the
    search space of the Krylov subspace method, which is at the same time the
    Galerkin test space, is ``enlarged'' by a suitably chosen subspace. A
    typical goal is to add information about the problem to the search space
    that is slowly revealed in the Krylov subspace itself, e.g. eigenvectors
    corresponding to small eigenvalues.

    Deflation and augmentation techniques can be combined with conventional
    preconditioning techniques. Then the projection and augmentation parameters
    have to be adapted to the preconditioned matrix.  In this paper, we assume
    that equation \eqref{Axb} is already in preconditioned form, i.e.,
    $\bfA$ is the preconditioned matrix and $\bfb$ the preconditioned right-hand
    side.  Details of preconditioning techniques will thus not be addressed
    here.

    We will now give a brief overview of existing deflation and augmentation
    strategies.  For a more comprehensive presentation we refer to Section~9 of
    the survey article by Simoncini and Szyld~\cite{SimS07}.  The first
    deflation and augmentation techniques in the context of Krylov subspace
    methods appeared in the papers of Nicolaides~\cite{Nic87} and
    Dost\'al~\cite{Dos88}. Both proposed deflated variants of the \CG{}
    method~\cite{HesS52} to accelerate the speed of convergence for symmetric
    positive definite (spd) matrices $\bfA$ arising from discretized elliptic partial
    differential equations.  Since these early works deflation and augmentation
    have become widely used tools.  Several authors working in different fields
    of numerical analysis applied them to many Krylov subspace methods, and they
    use a variety of techniques to determine a deflation subspace.  A review of
    all applications is well beyond this introduction.  We concentrate in the
    following on some --- but not all --- key contributions.

    For nonsymmetric systems Morgan~\cite{Mor95} and also Chapman and
    Saad~\cite{ChaS97} extracted approximate eigenvectors of $\bfA$ from the
    Krylov subspace generated by the \GMRES{} method~\cite{SaaS86}, and then
    they augmented the Krylov subspace with these vectors; for related
    references we refer to~\cite{FraV01}.  A comparable approach in the context
    of the \CG{} method for spd matrices $\bfA$ was
    described by Saad, Yeung, Erhel, and Guyomarc'h~\cite{SaaYEG00}. De
    Sturler~\cite{Stu96} introduced the \GCRO{} method, which involves an outer
    \GCR{} iteration~\cite{EisES83,Elm82} and an inner deflated \GMRES{} method
    where the space used for deflation depends on the outer iteration.  This
    method has been extended to \GCROT{} in~\cite{Stu99} to incorporate
    truncation strategies when restarts are necessary.  In~\cite{Kol98}
    Kolotilina used a twofold deflation technique for simultaneously deflating
    the $r$ largest and the $r$ smallest eigenvalues by an appropriate deflating
    subspace of dimension $r$.  An analysis of acceleration strategies
    (including augmentation) for minimal residual methods was given by
    Saad~\cite{Saa97} and for restarted methods by Eiermann, Ernst and
    Schneider~\cite{EieES00}.  The latter work analyzes minimal residual (MR)
    and orthogonal residual (OR) methods in a general framework that allows
    approximations from arbitrary correction spaces.  By using multiple
    correction spaces forming a direct sum, several cases of augmentation and
    deflation are discussed. The analysis concentrates on (nearly)
    $\bfA$-invariant augmentation spaces.

    In~\cite{Mor05} Morgan proposed a block-\GMRES{} method for multiple
    right-hand sides that deflates approximated eigenvectors when \GMRES{} is
    restarted. A similar method for solving systems with multiple shifts and
    multiple right-hand sides has been introduced by Darnell, Morgan and
    Wilcox~\cite{DarMW08}. Giraud \ETAL~\cite{GirGPV10} recently developed a
    flexible \GMRES{} variant with deflated restarting where the preconditioner
    may vary from one iteration to the next.  In~\cite{OlsS10} Olshanskii and
    Simoncini studied spectral properties of saddle point matrices
    preconditioned with a block-diagonal preconditioner and applied a deflated
    \MINRES{} method to the resulting symmetric and indefinite matrix in order
    to alleviate the influence of a few small outlying eigenvalues. Theoretical
    results for deflated \GMRES{} based on an exactly $\bfA$-invariant subspace
    have been presented in~\cite{YeuTV10}.

    In addition to deflation/augmentation spaces based on approximative
    eigenvectors, other choices have been studied.  Mansfield~\cite{Man90}
    showed how Schur complement-type domain decomposition methods can be seen as
    a series of deflations.  Nicolaides~\cite{Nic87} constructed a deflation
    technique based on piecewise constant interpolation from a set of $r$
    subdomains, and he pointed out that deflation might be effectively used with
    a conventional preconditioner. In~\cite{Man91} Mansfield used the same
    ``subdomain deflation'' in combination with damped Jacobi smoothing, and
    obtained a preconditioner that is related to the two-grid method.  Baker,
    Jessup and Manteuffel~\cite{BakJM05} proposed a \GMRES{} method that is
    augmented upon restarts by approximations to the error.

    In~\cite{NabV04, NabV06, NabV08} Nabben and Vuik described similarities
    between the deflation approach and domain decomposition methods for
    arbitrary deflation spaces. This comparison was extended to multigrid
    methods in~\cite{TanNVE09,TanMNV10}.

    This brief survey indicates that in principle deflation or augmentation can
    be incorporated into every Krylov subspace method. However, some methods may
    suffer from mathematical shortcomings like breakdowns or numerical problems
    due to round-off errors.  The main goal of this paper is not to add further
    examples to the existing collection, but to introduce first a suitable
    framework for a whole family of such augmented and deflated methods
    (Section~\ref{secdeflaug}) and then to prove some results just assuming this
    framework (Section~\ref{secequivthm}).  The framework focuses on Krylov
    subspace methods whose residuals satisfy a certain Galerkin condition with
    respect to a true or formal inner product.  In Section~\ref{secequivthm}, we
    mathematically characterize the equivalence of two approaches for realizing
    such methods and discuss them along with potential pitfalls.  We then
    discuss known approaches to deflate \CG{} (Section \ref{secHpdA}), \GMRES{}
    (Section \ref{secgeneralA}), and \MINRES{} (Section \ref{secHermitianA}) in
    the light of our general equivalence theorem.  Among other results, this
    will show that a recent version of deflated \MINRES{}, which is part of the
    \RMINRES{} (``recycling'' \MINRES{}) method suggested by Wang, de Sturler
    and Paulino \cite{WanSP07}, can break down and how these breakdowns can be
    avoided by either adapting the right-hand side or the initial guess.
    We do not focus on specific implementations or algorithmic
    details but on the mathematical theory of these methods. For the numerical
    application in Section~\ref{secnumexp} we draw on the most robust \MINRES{}
    implementation that is available.

	\section{A framework for deflated and augmented Krylov methods}
	\label{secdeflaug}
	
    In this section we describe a general framework for deflation and
    augmentation, which simultaneously covers several Krylov subspace methods
    whose residuals satisfy a Galerkin condition.  Given an initial guess
    $\bfx_0\in\CC^N$, a positive integer $n$, an $n$-dimensional subspace
    $\calS_n$ of $\CC^N$, and a nonsingular matrix $\bfB\in\CC^{N\times N}$, let
    us first consider an approximation $\bfx_n$ to the solution $\bfx$ of the
    form
	\begin{align}
		\bfx_n \in \bfx_0 + \calS_n, \label{xn}
	\end{align}
    so that the corresponding residual
    \[
        \bfr_n \DEF \bfb - \bfA \bfx_n \in \bfr_0 + \bfA \calS_n
    \]
    satisfies
	\begin{equation}
		\bfr_n \perp \bfB \calS_n. \label{rn}
	\end{equation}

    If $\bfB^\H\bfA$ is Hermitian and positive definite (Hpd) then $\bfB^\H\bfA$
    induces an inner product $\ipdots[\bfB^\H\bfA]$, a corresponding norm
    $\|\cdot\|_{\bfB^\H\bfA}$, and an orthogonality $\perp_{\bfB^\H\bfA}$.
    Imposing equations \eqref{xn} and \eqref{rn} can then be seen to be
    equivalent to solving the following minimization problem:
    \begin{equation}
        \label{min}
        \text{find}\quad \bfx_n\in\bfx_0+\calS_n \quad\text{s.t.}\quad
        \|\bfx-\bfx_n\|_{\bfB^\H\bfA}
        =\min_{\bfy\in\bfx_0+\calS_n} \|\bfx-\bfy\|_{\bfB^\H\bfA}.
    \end{equation}
    Note that due to $\bfr_n = \bfA (\bfx - \bfx_n)$ the condition \eqref{rn}
    can be written as orthogonality condition for the error $\bfx - \bfx_n$:
    \begin{equation}
		(\bfx - \bfx_n) \perp_{\bfB^\H\bfA} \calS_n. \label{xn-x}
	\end{equation}

    The following two cases where $\bfB^\H\bfA$ is Hpd are of particular
    interest:
    \par\medskip\noindent
    \begin{enumerate}
        \renewcommand{\labelenumi}{(\arabic{enumi})}
        \item $\bfB=\bfI$ if $\bfA$ itself is Hpd;
        \item $\bfB=\bfA$ for general nonsingular $\bfA$.
    \end{enumerate}
    \par\medskip\noindent
    The case (1) is the one where \eqref{rn} is a typical Galerkin condition:
    $\bfA$ is Hpd and the residual $\bfr_n$ is orthogonal to the linear search
    space $\calS_n$ for $\bfx_n - \bfx_0$.  In \eqref{min} we then have
    \begin{equation}
       \label{Ainverse-norm}
       \|\bfx-\bfx_n\|_{\bfA} =  \|\bfr_n\|_{\bfA^{-1}} ,
    \end{equation}
    so while the error is minimal in the $\bfA$--norm, the residual is minimal
    in the ${\bfA^{-1}}$--norm.

    In this paper we will refer to \eqref{rn} also in the case (2) as a
    Galerkin condition, because the search space and the test space are still
    essentially the same. However, in this case
    \begin{equation}
       \label{2-norm-res}
       \|\bfx-\bfx_n\|_{\bfA^\H\bfA} =  \|\bfr_n\|_2 ,
    \end{equation}
    so \eqref{rn} implies that the $2$--norm of the residual is minimized.
    Consequently, in both cases a minimization property holds.

    If the search space $\calS_n$ is the $n$-th Krylov subspace generated by
    $\bfA$ and the initial residual $\bfr_0:=\bfb-\bfA\bfx_0$, i.e., if
    \begin{equation}
       \label{calKn(A)}
       \calS_n = \calK_n \left( \bfA, \bfr_0 \right)
       \DEF \Span\{\bfr_0, \bfA\bfr_0, \dots, \bfA^{n-1}\bfr_0\},
    \end{equation}
    then, in the case (1), conditions \eqref{xn}--\eqref{rn} mathematically characterize
    the \CG{} method~\cite{HesS52}.  It is the prototype of an {\em Orthogonal
    Residual (OR) method} characterized by \eqref{xn} and \eqref{rn} with
    $\bfB=\bfI$.

    In the case (2), conditions \eqref{xn}--\eqref{rn} with $\calS_n=\calK_n
    \left( \bfA, \bfr_0 \right)$ mathematically characterize the
    \GCR~\cite{EisES83} and \GMRES{}~\cite{SaaS86} methods and, for Hermitian
    $\bfA$, the \MINRES{}~\cite{PaiS75} method.  If $\bfA$ is even Hpd, we can
    resort to Stiefel's Conjugate Residual (\CR{}) method \cite{Sti55}.  All
    these are prototype {\em Minimal Residual (MR) methods} characterized by
    \eqref{xn} and \eqref{rn} with $\bfB=\bfA$.

    Orthogonal Residual and Minimal Residual methods often come in pairs defined
    by the properties of $\bfA$, the Krylov search space, and, to some extent,
    the fundamental structure of the algorithms.  Examples of such pairs are
    \CG/\CR, \GCG/\GCR, \FOM/GMRES, and \CGNE/\CGNR.  It has been pointed out
    many times, see, e.g., \cite{Bro91,Cul95,EieE01,EieES00,GutRoz01b}, that the
    residuals of these pairs of OR/OM methods and in particular the residual
    norms are related in a simple fashion.

    A fact related to the OR/MR residual connection is that the iterates and
    residuals of an MR method can be found from those of the corresponding OR
    method by a smoothing process introduced by Sch\"onauer; see
    \cite{Wei90,Wei94b,GutRoz01b,GutRoz02}.  The reverse process also exists
    \cite{GutRoz01b}.  Again these processes hold for the residuals of the
    deflated system and, since they are identical, for those of the explicit
    augmentation approach.

    If $\bfB^\H\bfA$ is not Hpd, the minimization property \eqref{min} no longer
    makes sense, but we may still request that the orthogonality condition
    \eqref{rn} or, equivalently, \eqref{xn-x} hold. Resulting algorithms may
    then break down since an approximate solution $\bfx_n$ satisfying the
    conditions may not exist for some $n$.  Nevertheless, such methods are
    occasionally applied in practice.  In particular, the choice
    \par\medskip\noindent
    \begin{enumerate}
        \renewcommand{\labelenumi}{(\arabic{enumi})}
        \addtocounter{enumi}{2}
        \item $\bfB=\bfI$ and $\bfA$ nonsingular
    \end{enumerate}
    \par\medskip\noindent
    covers the full orthogonalization method (FOM) of Saad \cite{Saa81,Saa03},
    which is sometimes also referred to as Arnoldi method for linear algebraic
    systems.

    For minimizing the error $\bfx_n - \bfx$ in the 2-norm one has to choose
    \par\medskip\noindent
    \begin{enumerate}
        \renewcommand{\labelenumi}{(\arabic{enumi})}
        \addtocounter{enumi}{3}
        \item $\bfB=\bfAmH$ and $\bfA$ nonsingular.
    \end{enumerate}
    \par\medskip\noindent
    Since multiplication by $\bfAmH$ is not feasible, these methods only work
    for particular search spaces; the simplest choice is
    \begin{equation}
       \label{calKn(A^H)}
       \calS_n = \bfAH \calK_n \left( \bfAH, \bfr_0 \right).
    \end{equation}
    Unlike the normal Krylov search space of \eqref{calKn(A)}, this one has the
    drawback that the (exact) solution of the system need not be in one of these
    spaces, i.e., even in exact arithmetic convergence is not guaranteed.  One
    interesting example based on this choice is the Generalized Minimum Error
    (\GMERR{}) method of Weiss \cite{Wei94a}.  Earlier, for spd matrices, such a
    method was proposed by Friedman \cite{Fri63}, and an alternative algorithm
    was mentioned by Fletcher \cite{Fle76}.  Symmetric indefinite systems can be
    treated in this way with the \SYMMLQ{} algorithm of Paige and Saunders
    \cite{PaiS75}; see also Freund \cite{Fre90} for a review of methods
    featuring this optimality criterion and yet another algorithm called ME to
    achieve it.

    Finally, we can easily incorporate the \CGNR{} method \cite{HesS52} for
    solving overdetermined linear systems in the setting of \eqref{xn} and
    \eqref{rn} by choosing the appropriate Krylov search space.  Given such a
    system
    \begin{equation}
       \label{Ex=f}
       \bfE \bfx = \bff
    \end{equation}
    with a full-rank $M \times N$--matrix $\bfE$ (where $M \geq N$), the
    corresponding normal equations are $\bfE^\H \bfE \bfx = \bfE^\H \bff$, i.e.,
    $\bfA \bfx = \bfb$ with $\bfA \DEF \bfE^\H \bfE$ and $\bfb \DEF \bfE^\H
    \bff$.  Since $\bfA$ is Hpd, we can apply the \CG{} method which corresponds
    to the case (1) and
    \begin{equation}
       \label{calKn(E^HE)}
       \calS_n = \calK_n \left( \bfE^\H \bfE, \bfE^\H \bfs_0 \right)
    \end{equation}
    with $\bfs_0\DEF \bff - \bfE\bfx_0$.
    In this situation we have to distinguish between the residuals $\bfr_n \DEF
    \bfb - \bfA \bfx_n = \bfE^\H \bff - \bfE^\H \bfE \bfx_n$ of the normal
    equations and the residuals $\bfs_n \DEF \bff - \bfE \bfx_n$ of the given
    system \eqref{Ex=f}. The \CGNR\ method allows one to keep track of both.
    The latter residuals satisfy
    \begin{equation}
       \label{trueres}
       \bfs_n \in \bfs_0 + \bfE \calK_n (\bfE^\H \bfE, \bfE^\H \bff), \qq
       \bfs_n \perp \bfE \calK_n (\bfE^\H \bfE, \bfE^\H \bff),
    \end{equation}
    and they can be seen to minimize the 2-norm of $\bfs_n$. Note that it can be
    viewed as an MR method with a possibly non-square $\bfB = \bfE$; see
    \cite{GutRoz01b}.

    A method that also fits into our framework, though with some modifications,
    is the \CGNE{} method, also called Craig's method \cite{Cra55}, which can
    also be used for solving underdetermined linear algebraic
    systems \eqref{Ex=f} with a full-rank $M \times N$--matrix $\bfE$.
    The search space for $\bfx_n$ in this case is \eqref{calKn(E^HE)}, but
    the Galerkin condition becomes $\bfs_n \perp \calK_n(\bfE \bfE^\H,\bfs_0)$.

    Since we are aiming at a {\em general framework}, let us for the moment
    consider an arbitrary, possibly singular matrix $\bfAhat\in\CC^{N\times N}$
    and an arbitrary vector $\bfvhat\in\CC^N$, such that the Krylov subspace
    $\calK_n (\bfAhat, \bfvhat)$ has dimension $n$.

    Instead of a search space of the form $\calS_n=\calK_n \left( \bfA,
    \bfr_0\right)$ we focus from now on {\em augmented} Krylov subspaces of
    the form
	\begin{align}
        \calS_n \DEF \calK_n(\bfAhat, \bfvhat) + \calU. \label{Sn}
	\end{align}
    We suppose that $\calU$ has dimension~$k$, $0<k<N$, and denote by $\bfU \in
    \CC^{N \times k}$ a matrix whose columns form a basis of $\calU$, and by
    $\bfV_n \in \CC^{N \times n}$ one whose columns form a basis of
    $\calK_n(\bfAhat,\bfvhat)$, so that \eqref{xn} can be written as
	\begin{equation}\label{xnVU}
		\bfx_n = \bfx_0 + \bfV_n \bfy_n + \bfU \bfu_n
	\end{equation}
    for some vectors $\bfy_n\in\CC^n$ and $\bfu_n\in\CC^k$.  Of course, $\bfU$
    may be redefined when an algorithm like \GMRES\ is restarted, but we will
    not account for that in our notation.

    Assuming the general structure of the search space $\calS_n$ in \eqref{Sn} we will now
    investigate {\em augmented} Galerkin-type methods that still satisfy
    \eqref{xn} and \eqref{rn}.

	\section{A general equivalence theorem}
	\label{secequivthm}

    Our goal in this section is to show that augmentation can be achieved
    either explicitly as in \eqref{Sn}, or implicitly, namely by
    projecting the residuals appropriately and correcting the approximate
    solutions in a final step. Our main result is stated in
    Theorem~\ref{thmequi} below.

    In order to satisfy \eqref{rn}, the residual $\bfr_n=\bfb-\bfA\bfx_0$ must
    be orthogonal to both $\bfB\calK_n(\bfAhat, \bfvhat)$ and $\bfB\calU$, hence
    it must satisfy the pair of orthogonality conditions
	\begin{equation}\label{rncond}
        \bfr_n \perp \bfB\, \calK_n(\bfAhat, \bfvhat) \quad\mbox{and} \quad
		\bfr_n \perp \bfB\, \calU.
	\end{equation}
    Let us concentrate on the second condition of~\eqref{rncond}, which can be written as
	$$		
    \bf0 = \bfU^\H \bfB^\H \bfr_n
	= \bfU^\H \bfB^\H \left(\bfr_0 - \bfA \bfV_n \bfy_n - \bfA \bfU \bfu_n\right)
	= \bfU^\H \bfB^\H \left(\bfr_0 - \bfA \bfV_n \bfy_n\right) - \bfEB \bfu_n ,
	$$
	where
	\begin{equation}\label{EB}
		\bfE_\bfB \DEF \bfU^\H \bfB^\H \bfA \bfU\in\CC^{k\times k}.
	\end{equation}
    Clearly, if $\bfB^\H \bfA$ is Hpd, then $\bfE_\bfB$ is Hpd too --- though in
    a smaller space --- and thus nonsingular.  In the following derivation,
    $\bfB^\H \bfA$ need not be Hpd, but we must then assume that $\bfE_\bfB$ is
    nonsingular.  Then the second orthogonality condition is equivalent to
	\begin{equation}
		\bfu_n = \bfE_\bfB^{-1} \bfU^\H \bfB^\H \left(\bfr_0 - \bfA \bfV_n \bfy_n\right). \label{un}
	\end{equation}
    Substituting this into \eqref{xnVU} gives
	\begin{align}
		\bfx_n &= \bfx_0 + \bfV_n \bfy_n + \bfU \left( \bfE_\bfB^{-1} \bfU^\H \bfB^\H
		\left(\bfr_0 - \bfA \bfV_n \bfy_n \right)\right) \notag \\
		&= \left( \bfI - \bfU \bfE_\bfB^{-1} \bfU^\H \bfB^\H \bfA\right)
		\left( \bfx_0 + \bfV_n \bfy_n \right) + \bfU \bfE_\bfB^{-1} \bfU^\H \bfB^\H \bfb,
		\label{xnPfull} \\
		\bfr_n &= \bfr_0 - \bfA \bfV_n \bfy_n - \bfA \bfU \left( \bfE_\bfB^{-1} \bfU^\H \bfB^\H
		\left(\bfr_0 - \bfA \bfV_n \bfy_n \right)\right) \notag\\
		&= \left( \bfI - \bfA \bfU \bfE_\bfB^{-1} \bfU^\H \bfB^\H \right)
		\left( \bfr_0 - \bfA \bfV_n \bfy_n \right).
		\label{rnPfull}
	\end{align}
	To simplify the notation we define the ($N\times N$)--matrices
	\begin{align}
		\bfMB  &\DEF \bfU \bfE_\bfB^{-1} \bfU^\H =
		\bfU \left( \bfU^\H \bfB^\H \bfA \bfU \right)^{-1} \bfU^\H,
		\nonumber\\
		\bfPB  &\DEF \bfI - \bfA \bfMB \bfB^\H, \label{QPP}\\
		\bfQB &\DEF \bfI - \bfMB \bfB^\H \bfA.\nonumber
	\end{align}
    Using these matrices the equations \eqref{xnPfull} and \eqref{rnPfull} take
    the form
	\begin{align}
		\bfx_n &= \bfQB \left(\bfx_0 + \bfV_n \bfy_n\right) + \bfMB \bfB^\H \bfb , \label{xnP} \\
		\bfr_n &= \bfPB \left( \bfr_0 - \bfA \bfV_n \bfy_n \right) \label{rnP}.
	\end{align}
	Note that imposing the second orthogonality condition in \eqref{rncond} on
	the residual $\bfr_n$ has determined the vector $\bfu_n$, which has therefore
	``disappeared'' in \eqref{xnP}--\eqref{rnP}.
	We next state some basic properties of the matrices $\bfPB$ and $\bfQB$. The
	proof of these properties is straightforward, and is therefore omitted.
	\begin{lemma}
		\label{lemproj}
        Let $\bfA,\bfB\in\CC^{N\times N}$ and $\bfU\in\CC^{N\times k}$ be such
        that $\bfE_\bfB\DEF \bfU^\H\bfB^\H\bfA\bfU$ is nonsingular (which
        implies that $\rank\bfU=k$). Then the matrices in \eqref{QPP} are
        well defined and the following statements hold:
		\begin{enumerate}
			\item \label{lemprojPAU} $\bfPB^2 = \bfPB$, $\bfPB \bfA \bfU = \bf0$, and
				$\bfU^\H \bfB^\H \bfPB = \bf0$, i.e., $\bfPB$ is the projection onto
				$\left( \bfB \calU \right)^\perp$ along $\bfA \calU$.
			\item \label{lemprojPtU} $\bfQB^2 = \bfQB$, $\bfQB \bfU = \bf0$,
                and $\bfU^\H \bfB^\H \bfA \bfQB = \bf0$, i.e., $\bfQB$ is the
                projection onto $\left(\bfA^\H \bfB \calU \right)^\perp$ along $\calU$.
			\item \label{PAAPt} $\bfPB \bfA = \bfPB \bfA \bfQB = \bfA \bfQB$.
            \item $\bfPA = \bfPA^\H$, i.e., $\bfPA$ is an orthogonal projection.
		\end{enumerate}
	\end{lemma}
    It remains to impose the first orthogonality condition in \eqref{rncond},
    which will determine the vector $\bfy_n$. To this end, let
    $$
        \bfxhat_n \DEF \bfx_0 + \bfV_n \bfy_n\in \bfx_0 + \calK_n(\bfAhat,
        \bfvhat),
    $$
    so that by \eqref{xnP} $\bfx_n = \bfQB \bfxhat_n + \bfMB \bfB^\H \bfb$.
    Using the definition of $\bfPB$ in \eqref{QPP} and statement~\ref{PAAPt} of
    Lemma~\ref{lemproj}, this orthogonality condition reads
	$$
        \bfr_n = \bfb - \bfA \bfx_n
        = \bfb - \bfA\bfQB\bfxhat_n - \bfA\bfMB\bfB^\H\bfb
        = \bfPB (\bfb - \bfA \bfxhat_n) \perp \bfB \calK_n(\bfAhat, \bfvhat).
    $$
	We summarize these considerations in the following theorem.
    \begin{theorem}
        \label{thmequi}
        Let the assumptions of Lemma~\ref{lemproj} hold and let
        $\bfAhat\in\CC^{N\times N}$, $\bfvhat\in\CC^N$ and $n\in\NN$ be such
        that the Krylov subspace $\calK_n(\bfAhat,\bfvhat)$ has dimension $n$.
        Furthermore, let $\bfb,\bfx_0\in\CC^N$ be arbitrary.

        Then, with $\calU\DEF\Image(\bfU)$ and the definitions from \eqref{QPP}
        the following two pairs of conditions,
        \begin{align}
			\begin{aligned}
                \bfx_n &\in \bfx_0 + \calK_n(\bfAhat, \bfvhat) + \calU , \\
				\bfr_n &\DEF \bfb - \bfA \bfx_n \perp \bfB \calK_n(\bfAhat,
                \bfvhat) + \bfB \calU,
			\end{aligned}
			\label{ansatz1}
		\end{align}
		and
		\begin{align}
			\begin{aligned}
                \bfxhat_n &\in \bfx_0 + \calK_n(\bfAhat, \bfvhat) , \\
				\bfrhat_n &\DEF \bfPB (\bfb - \bfA \bfxhat_n) \perp \bfB
                \calK_n(\bfAhat, \bfvhat).
			\end{aligned}
			\label{ansatz2}
		\end{align}
		are equivalent for $n\geq1$ in the sense that
		\begin{equation}\label{xnP'}
			\bfx_n = \bfQB \bfxhat_n + \bfMB \bfB^\H \bfb ,
			\quad\mbox{and}\quad
			\bfr_n = \bfrhat_n.
		\end{equation}
    \end{theorem}

    We call \eqref{ansatz1} the {\em explicit} deflation and augmentation
    approach because the augmentation space $\calU$ is explicitly included in
    the search space.  The equivalent conditions \eqref{ansatz2} show that the
    explicit inclusion of $\calU$ can be omitted when instead we first construct
    the iterate $\bfxhat_n \in \bfx_0 + \calK_n(\bfAhat, \bfvhat)$ so that the
    projected residual $\bfrhat_n=\bfPB(\bfb - \bfA \bfxhat_n)$ satisfies the
    given orthogonality condition and then apply the affine correction
    \eqref{xnP'} to $\bfxhat_n$, whose projected residual equals the one of
    $\bfx_n$.  We call this second option the {\em implicit} deflation and
    augmentation approach.

    Note that the theorem makes no assumption on relations between $\bfAhat$,
    $\bfvhat$, and $\calU$.  The only assumption on the augmentation space
    $\calU$ is that the matrix $\bfU^\H\bfB^\H\bfA\bfU$ is nonsingular.
    (Clearly, if this holds for one basis of $\calU$ it holds for all.)
    Moreover, in the theorem $\bfAhat$ and $\bfvhat$ are arbitrary except for
    the assumption that $\calK_n(\bfAhat,\bfvhat)$ has dimension $n$.

    In practice, $\bfAhat$ and $\bfvhat$ should be somehow related to $\bfA$,
    however. One specific choice is suggested by Theorem~\ref{thmequi}, in
    particular \eqref{ansatz2}. If
    \[
        \bfAhat\DEF\bfPB\bfA,\quad
        \bfvhat\DEF\bfrhat_0\DEF\bfPB\bfr_0=\bfPB(\bfb-\bfA\bfx_0)
        \quad \text{and}\quad
        \bfbhat\DEF\bfPB\bfb
    \]
    then \eqref{ansatz2} becomes
    \begin{align}
        \begin{aligned}
            \bfxhat_n &\in \bfx_0 + \calK_n(\bfAhat, \bfrhat_0) , \\
            \bfrhat_n &\DEF \bfbhat - \bfAhat \bfxhat_n \perp \bfB
            \calK_n(\bfAhat, \bfrhat_0),
        \end{aligned}
        \label{ansatz2_sugg}
    \end{align}
    which is a formal Galerkin condition for the (consistent and singular) {\em
    deflated system} $\bfAhat \bfxhat=\bfbhat$.  Based on the Jordan form of
    $\bfA$ we show in the following theorem how the Jordan form of
    $\bfAhat=\bfPB\bfA$ looks like when (1) $\calU$ is a right invariant
    subspace or (2) $\bfB\calU$ is a left invariant subspace of $\bfA$.

    \begin{theorem}
        \label{thm:spectrum}
        Suppose that the matrix $\bfA\in\CC^{N\times N}$ has a partitioned
        Jordan decomposition of the form
        \begin{equation}\label{jordanA}
            \bfA
            = \bfS \bfJ \bfS^{-1}
            = \Matrix{cc}{\bfS_1 & \bfS_2}
            \Matrix{cc}{\bfJ_1 & \bf0 \\ \bf0 & \bfJ_2}
            \Matrix{c}{\WH\bfS_1^\H \\ \WH\bfS_2^\H}\,,
        \end{equation}
        where $\bfS_1,\WH\bfS_1\in\CC^{N\times k}$,
        $\bfS_2,\WH\bfS_2\in\CC^{N\times (N-k)}$,
        $\bfJ_1\in\CC^{k\times k}$, and $\bfJ_2\in\CC^{(N-k)\times(N-k)}$.
        Then the following assertions hold:
        \begin{enumerate}
            \item[(1)] If $\calU=\Image(\bfS_1)$, $\bfU\in\CC^{N\times k}$ is
                any matrix satisfying $\Image(\bfU)= \calU$ and $\bfU^\H\bfB^\H
                \bfA\bfU$ is nonsingular, then
                \begin{align}
                    \label{jordanPA-right}
                    \bfAhat = \bfPB\bfA
                    = \Matrix{cc}{\bfU & \bfPB \bfS_2}
                    \Matrix{cc}{\bf0 & \bf0 \\ \bf0 & \bfJ_2}
                    \Matrix{cc}{\bfU & \bfPB \bfS_2}^{-1}\\
                    \text{with}\q
                    \Matrix{cc}{\bfU & \bfPB \bfS_2}^{-1} =
                    \Matrix{cc}{\bfB\bfU(\bfU^\H\bfB\bfU)^{-1} & \WH\bfS_2}^\H.
                    \notag
                \end{align}

            \item[(2)] If $\bfB\calU=\Image(\WH\bfS_1)$, $\bfU\in\CC^{N\times
                k}$ is any matrix satisfying $\Image(\bfU)=\calU$ and
                $\bfU^\H\bfB^\H \bfA\bfU$ is nonsingular, then
                \begin{align}
                    \label{jordanPA-left}
                    \bfAhat = \bfPB\bfA
                    = \Matrix{cc}{\bfU & \bfS_2}
                    \Matrix{cc}{\bf0 & \bf0 \\ \bf0 & \bfJ_2}
                    \Matrix{cc}{\bfU & \bfS_2}^{-1}\\
                    \text{with}\q
                    \Matrix{cc}{\bfU & \bfS_2}^{-1}=
                    \Matrix{cc}{\bfB\bfU(\bfU^\H\bfB\bfU)^{-1} &
                    \bfQB^\H\WH\bfS_2}^\H.
                    \notag
                \end{align}
        \end{enumerate}
        In particular, in both cases the spectrum $\Lambda(\bfAhat)$ of
        $\bfAhat$ is given by $\Lambda(\bfAhat)=\{0\}\cup\Lambda(\bfJ_2)$.
    \end{theorem}

    \begin{proof}
        (1) From Lemma~\ref{lemproj} we can see that
        \[
            \bfPB\bfA\bfU=\bf0
            \q\text{and}\q
            \left(\bfB\bfU(\bfU^\H\bfB\bfU)^{-1}\right)^\H\bfPB\bfA=\bf0.
        \]
        By construction, there exists a nonsingular matrix $\bfR\in\CC^{k\times
        k}$ with $\bfA\bfU=\bfU\bfR$. Hence $\bfPB=\bfI - \bfU
        (\bfU^\H\bfB^\H\bfU)^{-1}\bfU^\H\bfB^\H$ and from
        $\WH\bfS_2^\H\bfS_1=\bf0$ we conclude that $\WH\bfS_2^\H\bfU=\bf0$ and
        thus
        \[
            \WH\bfS_2^\H\bfPB\bfA=\WH\bfS_2^\H\bfA=\bfJ_2\WH\bfS_2^\H.
        \]
        Furthermore,
        \[
            \bfPB\bfA (\bfPB\bfS_2) =
            \bfPB \bfA\bfS_2 -
            \bfPB\bfA\bfU(\bfU^\H\bfB^\H\bfU)^{-1}\bfU^\H\bfB^\H\bfS_2=
            (\bfPB\bfS_2)\bfJ_2
        \]
        and the proof of (1) is complete after recognizing that
        \[
            \Matrix{c}{ (\bfU^\H\bfB^\H\bfU)^{-1}\bfU^\H\bfB^\H \\
            \WH\bfS_2^\H}
            \Matrix{cc}{\bfU & \bfPB \bfS_2} = \Matrix{cc}{\bfI & \bf0\\
            \bf0 & \bfI}.
        \]
        The proof of (2) is analogous to (1).
    \end{proof}

    Results like the previous theorem motivate the term ``deflation'', which
    means ``making something smaller'', since the multiplication with the
    operator $\bfPB$ ``removes'' certain eigenvalues from the operator $\bfA$ by
    ``moving them to zero''.  Special cases of the results shown in
    Theorem~\ref{thm:spectrum} have appeared in the literature: in particular,
    for spd matrices $\bfA$ and $\bfB=\bfI$ in the
    works of Frank and Vuik~\cite{FraV01} and Nabben and
    Vuik~\cite{NabV04,NabV06}, and for nonsymmetric $\bfA$ and $\bfB=\bfI$ in
    the articles by Erlangga and Nabben~\cite{ErlN08,ErlN08a} and Yeung, Tang
    and Vuik~\cite{YeuTV10}.

	\section{Hermitian positive definite matrices and \CG}
	\label{secHpdA}

    This section presents some well-known results for deflation and augmentation
    techniques within the framework described in Section~\ref{secdeflaug} in the
    case where $\bfA$ is Hpd.  The first proposed deflated Krylov subspace
    methods for Hpd matrices are the deflated \CG{} variants of
    Nicolaides~\cite{Nic87} and Dost\'{a}l~\cite{Dos88}.  With a full-rank
    matrix $\bfU\in\CC^{N\times k}$, $\calU=\Image(\bfU)$ and $\bfB=\bfI$ both
    essentially apply the \CG{} method to the {\em deflated system}
	\begin{equation}\label{PIAxPIb}
        \bfAhat \bfxhat=\bfbhat,\quad\mbox{where}\quad
        \bfAhat \DEF \bfPI\bfA,\quad \bfbhat\DEF\bfPI \bfb.
	\end{equation}
    Here, $\bfPI=\bfI -\bfA \bfU(\bfU^\H\bfA\bfU)^{-1}\bfU^\H$ is the projection
    onto $\calU^\perp$ along $\bfA\calU$ as defined in \eqref{QPP} when $\bfB =
    \bfI$.  Moreover, $\bfQI = \bfPI^\H$ is then the projection onto
    $(\bfA\calU)^\perp$ along $\calU$.  Note that all the matrices in
    \eqref{QPP} are well defined because $\bfE_\bfI=\bfU^\H\bfA\bfU$ is Hpd if
    $\bfA$ is Hpd.  Clearly, the deflated matrix $\bfAhat$ is Hermitian but
    singular, since $\bfPI$ is a nontrivial projection if $0<k<N$.  In fact,
    this matrix $\bfAhat$ is positive semi-definite, since
	\begin{eqnarray*}
		\bfv^\H \bfAhat \bfv &=& \bfv^\H\bfPI \bfA\bfv = \bfv^\H\bfPI^2 \bfA\bfv=
		\bfv^\H\bfPI (\bfPI\bfA)\bfv= \bfv^\H\bfPI (\bfPI\bfA)^\H\bfv
		= \bfv^\H\bfPI \bfA\bfPI^\H\bfv\geq 0
	\end{eqnarray*}
    holds for any $\bfv \in \CC^N$.  The system \eqref{PIAxPIb} is consistent
    since it results from a left-multiplication of the nonsingular system
    $\bfA\bf x=\bfb$ by $\bfPI$.  (We note that in \cite{Nic87,Dos88} the
    application of the projection $\bfPI$ to $\bfb$ is carried out implicitly by
    adapting the initial guess such that the initial residual is orthogonal to
    $\calU=\Image(\bfU)$.) The solution $\bfx$ of $\bfA\bfx=\bfb$ thus also
    solves $\bfAhat\bfx=\bfbhat$, but in Eqn.~\eqref{PIAxPIb} we replaced $\bfx$
    by $\bfxhat$ to indicate the non-uniqueness of the solution.  In fact, the
    general solution is $\bfxhat = \bfx + \bfh$ with $\bfh \in \calU$ since
    $\bfPI \bfA \bfh = \bfA \bfQI \bfh = \bf0$ if and only if $\bfQI \bfh =
    \bf0$, that is, $\bfh \in \calU$; see statement~\ref{lemprojPtU} of
    Lemma~\ref{lemproj}.  The application of $\bfQI$ in the final correction
    \eqref{xnP'} will annihilate $\bfh$.
    Note that a deflation version including the final
    correction~\eqref{xnP'} is used by Frank and Vuik~\cite{FraV01} and Nabben
    and Vuik~\cite{NabV04,NabV06}.

    In the context of Hpd matrices the application of the \CG{} method to a
    deflated system like \eqref{PIAxPIb} is a commonly used technique; see,
    e.g.,~\cite{TanNVE09} for a survey of results.  Finally, we point out that
    $\bfAhat$ as defined in \eqref{PIAxPIb} is completely determined by $\bfA$
    and the choice of the space $\calU$.

    According to Nicolaides~\cite[Section~3]{Nic87} and
    Kaasschieter~\cite[Section~2]{Kaa88}, the \CG{} method is well defined
    (in exact arithmetic) for each step $n$ until it terminates with an exact
    solution, when it is applied to a consistent linear algebraic system with a
    real and symmetric positive semidefinite matrix.  This result easily
    generalizes to complex and Hermitian positive semidefinite matrices.

    Mathematically, the $n$-th step of the \CG{} method applied to the deflated
    system (\ref{PIAxPIb}) with the initial guess $\bfx_0$ and the corresponding
    initial residual $\bfrhat_0=\bfbhat-\bfAhat\bf x_0$ is characterized by the
    two conditions
	\begin{align*}
		\bfxhat_n &\in \bfx_0 + \calK_n(\bfAhat, \bfrhat_0) , \\
		\bfrhat_n &= \bfbhat - \bfAhat \bfxhat_n
		= \bfPI (\bfb - \bfA \bfxhat_n) \perp \calK_n(\bfAhat, \bfrhat_0),
	\end{align*}
    This is nothing but the set of conditions (\ref{ansatz2}) in
    Theorem~\ref{thmequi} with $\bfB=\bfI$.  In the sense of
    relation~\eqref{xnP'} these conditions have been shown to be equivalent to
    \eqref{ansatz1}, namely
	$$\begin{aligned}
		\bfx_n &\in \bfx_0 + \calK_n(\bfAhat, \bfrhat_0) + \calU , \\
		\bfr_n & = \bfb - \bfA \bfx_n \perp \calK_n(\bfAhat, \bfrhat_0) + \calU,
	\end{aligned}$$
    which is the starting point of the theory for the deflated \CG{} method
    developed in~\cite{SaaYEG00}, where the authors also showed the equivalence
    between \CG{} with explicit augmentation and \CG{} applied to the deflated
    system \eqref{PIAxPIb}; see Section~4 in \cite{SaaYEG00}, in particular
    Theorem 4.6.  In a partly similar treatment, Erhel and Guyomarc'h
    \cite{ErhG00} considered an augmented and deflated CG method where the
    augmentation space $\calU$ is itself a Krylov space.  It is worth mentioning
    that both Saad \ETAL \cite[Eqn.~(3.12)]{SaaYEG00} and Erhel and Guyomarc'h
    \cite[Eqn.~(3.2)]{ErhG00} use the initial correction
    \[
       \bfx_0 := \bfx_{-1} + \bfMI \bfr_{-1} \qq\text{with}\q \bfr_{-1} := \bfb - \bfA \bfx_{-1}
    \]
    to replace a given initial approximation $\bfx_{-1}$ by one with $\bfr_0
    \perp \calU$; in fact, it is easily seen that $\bfr_0 = \bfPI \bfr_{-1}$.

    The goal of deflation is to obtain a deflated matrix $\bfAhat$ whose
    ``effective condition number'' is smaller than the one of $\bfA$, for
    example by ``eliminating'' the smallest eigenvalues of $\bfA$.  A detailed
    analysis of spectral properties of $\bfPI \bfA$ and other projection-type
    preconditioners arising from domain decomposition and multigrid methods was
    carried out in~\cite{NabV06} and~\cite{TanNVE09}.  In particular, it was
    shown in these papers that the effective condition number of $\bfAhat$ is
    less than or equal to the condition number of $\bfA$ for {\em any}
    augmentation space $\calU$.  Moreover, if $\Lambda=\Lambda(\bfA)$ is the
    spectrum of $\bfA$ and $\calU$ is an $\bfA$-invariant subspace associated
    with the eigenvalues $\Theta=\{\theta_1,\ldots,\theta_k\}\subset\Lambda$,
    then the effective 2-norm condition number is
    \[
      \kappa_2(\bfAhat)=\frac{\max_{\lambda\in\Lambda\setminus\Theta}{\lambda}}
      {\min_{\lambda\in\Lambda\setminus\Theta}{\lambda}}.
    \]

    In summary, for {\em any} augmentation space $\calU$, the \CG{} method
    applied to the (singular) deflated system (\ref{PIAxPIb}) is well defined
    for any iteration step $n$, and it terminates with an exact solution
    $\bfxhat$ (in exact arithmetic). Once \CG{} has terminated with a solution
    $\bfxhat$ of the deflated system, we can obtain the uniquely defined
    solution of the original system using the final correction step
	$$\bfx = \bfQI \bfxhat + \bfMI \bfb$$
	(cf.\ \eqref{xnP'}), which indeed gives
	$$
        \bfA \bfx=\bfA\bfQI \bfxhat
    	+\bfA \bfMI \bfb=
    	{\bf P}_\bfI \bfA \bfxhat + \bfA \bfMI \bfb=
    	({\bf P}_\bfI+\bfA \bfMI) \bfb=\bfb.
    $$
    This computation is mathematically equivalent to an explicit use of
    augmentation.  Of course, in practice we stop the \CG{} iteration for the
    deflated system once the solution is approximated sufficiently accurately.
    We then use the computed approximation $\bfxhat_n$ and equation~\eqref{xnP'}
    from Theorem~\ref{thmequi} to obtain an approximation $\bfx_n$ of the
    solution of the given system $\bfA\bfx=\bfb$.  Note that, according to
    \eqref{xnP'}, the residual $\bfrhat_n = \bfbhat - \bfAhat \bfxhat_n$ of the
    projected system \eqref{PIAxPIb} is equal to the residual $\bfr_n = \bfb -
    \bfA \bfx_n$ of the original system \eqref{Axb}.

	\section{Non-Hermitian matrices and \GMRES}
	\label{secgeneralA}

    In this section we present mostly known results on applying versions of
    deflated \GMRES\ to a general nonsingular matrix $\bfA$.  We set $\bfB=\bfA$
    in the framework of Section~\ref{secdeflaug} and discuss some choices for
    $\bfAhat$, $\bfvhat$ and $\calU$.

    Morgan~\cite{Mor95} and also Chapman and Saad~\cite{ChaS97} presented
    variations of \GMRES{} that can be mathematically described by
    \eqref{ansatz1} with $\bfAhat=\bfA$ and $\bfvhat=\bfb-\bfA\bfx_0$. Hence
    they augmented the search space with an augmentation space $\calU$ but did
    neither deflate the matrix nor project the linear system onto a subspace of
    $\CC^N$.

    Erlangga and Nabben~\cite{ErlN08} used two matrices $\bfY,
    \bfZ\in\CC^{N\times k}$ to define the abstract deflation operator
    $\bfP_{\bfY\bfZ}\DEF\bfI - \bfA\bfZ(\bfY^\H\bfA\bfZ)^{-1}\bfY^\H$ for
    non-Hermitian matrices $\bfA$. Of course, this choice needs the assumption
    of nonsingularity of $\bfY^\H \bfA \bfZ$. Requiring $\bfY$ and $\bfZ$ to
    have full rank obviously is not sufficient. They then applied \GMRES{} to
    the deflated linear system $\bfP_{\bfY\bfZ}\bfA\bfxhat =
    \bfP_{\bfY\bfZ}\bfb$.

    De Sturler~\cite{Stu96} introduced the \GCRO{} method, which is a nested
    Krylov subspace method involving an outer and an inner iteration.  The outer
    method is the \GCR{} method \cite{EisES83,Elm82}, while the inner iteration
    uses the projection
	\[
        \bfPA=\bfI-\bfA\bfMA\bfA^\H=\bfI-\bfA
        \bfU(\bfU^\H\bfA^\H\bfA\bfU)^{-1}\bfU^\H\bfA^\H
    \]
    to apply several steps of \GMRES{} to the projected (or deflated) linear
    system
	\begin{equation}\label{AhxbhGM}
		\bfAhat \bfxhat =\bfbhat,\quad\mbox{where}\quad
		\bfAhat \ASS \bfPA \bfA,\quad \bfbhat \ASS \bfPA \bfb.
	\end{equation}
    In \GCRO{} the matrix $\bfU$ is determined from the corrections of the outer
    iteration.  Clearly, the matrix $\bfE_\bfA = \bfU^\H \bfA^\H \bfA \bfU$ is
    nonsingular for any matrix $\bfU\in\CC^{N\times k}$ with $\rank\bfU=k>0$, so
    that all matrices in \eqref{QPP} are well defined.  Note that the projection
    $\bfPA$ is equal to the abstract deflation operator $\bfP_{\bfY\bfZ}$
    of Erlangga and Nabben with the choice $\bfZ=\bfU$ and $\bfY=\bfA\bfU$.  For
    the application of $\bfPA$ only the matrix $\bfW\DEF\bfA\bfU$ is needed
    because $\bfPA=\bfI - \bfW(\bfW^\H\bfW)^{-1}\bfW^\H$.  De Sturler further
    simplified this in \cite[Section~2]{Stu96} to $\bfPA=\bfI-\bfC\bfC^\H$ by
    choosing a matrix $\bfC\in\CC^{N\times k}$ whose columns form an orthonormal
    basis of $\Image(\bfA\bfU$).

    Here we concentrate on the \GMRES{} method applied to the deflated system
    \eqref{AhxbhGM} and we first discuss some known results within the framework
    presented in Section~\ref{secdeflaug}.  Analogously to the approach for
    \CG{} described in the previous section, the deflated system \eqref{AhxbhGM}
    results from the given system $\bfA\bfx=\bfb$ by a left-multiplication with
    $\bfPA$ which projects onto $(\bfA\calU)^\perp$ along $\bfA\calU$.  Note
    that $\bfPA$ is an orthogonal projection, since $\bfPA$ is Hermitian.

    If we start \GMRES{} with an initial guess $\bfx_0$ and the corresponding
    initial residual $\bfrhat_0=\bfbhat-\bfAhat\bfx_0=\bfPA (\bfb-\bfA\bfx_0)$,
    then the iterate $\bfxhat_n$ and the residual $\bfrhat_n$ are characterized
    by the two conditions
	\begin{align*}
		\bfxhat_n \in \bfx_0 + \calK_n(\bfAhat, \bfrhat_0), \quad\mbox{and}\quad
		\bfrhat_n = \bfbhat - \bfAhat \bfxhat_n \perp \bfAhat \calK_n(\bfAhat, \bfrhat_0).
	\end{align*}
    If the columns of $\bfV_n$ form a basis of $\calK_n(\bfAhat, \bfrhat_0)$,
    then the second condition means that
	\begin{eqnarray*}
		\bf0 & = & \bfV_n^\H \bfAhat^\H \bfrhat_n = \bfV_n^\H \bfA^\H \bfPA^\H \bfrhat_n
		= \bfV_n^\H \bfA^\H \bfPA \bfPA \left( \bfb - \bfA \bfxhat_n \right) =
		\bfV_n^\H \bfA^\H \bfPA \left( \bfb - \bfA \bfxhat_n \right)\\
		&=& \bfV_n^\H \bfA^\H \bfrhat_n,
	\end{eqnarray*}
	or, equivalently,
	$$ \bfrhat_n \perp \bfA \calK_n(\bfAhat, \bfrhat_0).$$
    Note that here the Krylov subspace is multiplied with $\bfA$ instead of
    $\bfAhat$ and that this condition has precisely the form of the second
    condition in \eqref{ansatz2}.  Theorem~\ref{thmequi} now implies that the
    mathematical characterization of \GMRES{} applied to the deflated system
    $\bfAhat \bfxhat =\bfbhat$ is equivalent to the explicit use of
    augmentation, i.e., the conditions
	\begin{align}
		\bfx_n &\in \bfx_0 + \calK_n(\bfAhat, \bfrhat_0) + \calU \label{xnKnU} , \\
		\bfr_n &= \bfb - \bfA \bfx_n \perp \bfA \calK_n(\bfAhat, \bfrhat_0) + \bfA \calU, \label{rnAKnAU}
	\end{align}
	in the sense that
	\begin{equation}\label{xnP"}
		\bfx_n = \bfQA \bfxhat_n + \bfMA \bfA^\H \bfb ,\quad\mbox{and}\quad
		\bfr_n = \bfb -\bfA \bfx_n = \bfbhat - \bfAhat \bfxhat_n = \bfrhat_n.
	\end{equation}
    As mentioned in the beginning of Section~\ref{secdeflaug}, conditions
    \eqref{xnKnU}-\eqref{rnAKnAU} are equivalent to the minimization problem
    \begin{align*}
        \text{find}\quad \bfx_n\in\bfx_0+\calS_n \quad\text{s.t.}\quad
        \|\bfb-\bfA\bfx_n\|_2
        =\min_{\bfy\in\bfx_0+\calS_n }
        \|\bfb-\bfA\bfy\|_2
    \end{align*}
    with the search space $\calS_n=\calK_n(\bfAhat, \bfrhat_0) + \calU$. In the
    setting of \GCRO{}, where $\bfU$ is determined from the \GCR{} iteration,
    the equivalence between \GMRES{} applied to $\bfAhat\bfxhat=\bfbhat$ and the
    minimization problem with an explicitly augmented search space has already
    been pointed out by de Sturler~\cite[Theorem~2.2]{Stu96}.  The \GCRO{}
    method was extended to an arbitrary rank-$k$ matrix $\bfU$
    in~\cite[Section~2]{KilS06}.  In the case where $\calU$ is an
    $\bfA$-invariant subspace the equivalence is straightforward and has been
    pointed out by Eiermann, Ernst and Schneider~\cite[Lemma~4.3]{EieES00}.

    Again the deflated matrix $\bfAhat$ is singular, and we have to discuss
    whether the application of \GMRES{} to the deflated system yields (in exact
    arithmetic) a well-defined sequence of iterates that terminates with a
    solution. This turns out to be significantly more difficult than in the case
    of the \CG{} method. Properties of \GMRES{} applied to singular systems have
    been analyzed by de Sturler \cite{Stu96} and by Brown and
    Walker~\cite{BroW97}.  The following result is an extension of
    \cite[Theorem~2.6]{BroW97}.

	\begin{theorem}\label{thmsinggmres}
        Consider an arbitrary matrix $\bfAhat \in \CC^{N \times N}$ and a vector
        $\bfbhat \in \Image(\bfAhat)$ $($i.e., the linear algebraic system
        $\bfAhat \bfxhat = \bfbhat$ is consistent$)$.  Then the following two
        conditions are equivalent:
		\begin{enumerate}
            \item For every initial guess $\bfx_0 \in \CC^N$ the \GMRES{} method
                applied to the system $\bfAhat \bfxhat = \bfbhat$ is well defined
                at each iteration step $n$ and it terminates with a
                solution of the system.
            \item $\Ker (\bfAhat) \cap \Image (\bfAhat) = \left\{ \bf0
                \right\}$.
		\end{enumerate}
	\end{theorem}

	\begin{proof}
        It has been shown in~\cite[Theorem~2.6]{BroW97} that condition 2 implies
        condition 1.  We prove the reverse by contradiction. We assume that
        $\Ker (\bfAhat) \cap \Image(\bfAhat) \neq \left\{ \bf0 \right\}$, and we
        will construct an initial guess for which \GMRES{} does not terminate
        with the solution. For a nonzero vector $\bfy \in \Ker (\bfAhat) \cap
        \Image(\bfAhat)$ there exists a nonzero vector $\bfyhat \in \CC^N$, such
        that $\bfy=\bfAhat \bfyhat$, and since $\bfAhat \bfxhat = \bfbhat$ is
        consistent, there exists a vector $\bfxhat \in \CC^N$ with $\bfbhat =
        \bfAhat \bfxhat$.  Then the initial guess $\bfx_0 \ASS \bfxhat -
        \bfyhat$ gives $\bfr_0 = \bfbhat - \bfAhat \bfx_0 = \bfbhat - \bfAhat
        \bfxhat + \bfAhat \bfyhat = \bfy$.  But since $\bfy \in \Ker(\bfAhat)$,
        we obtain $\bfAhat \bfr_0 = \bf0$, so that the \GMRES{} method
        terminates at the first iteration with the approximation $\bfx_0$, for
        which $\bfr_0 = \bfy \neq \bf0$. Thus, for this particular initial guess
        $\bfx_0$ the \GMRES{} method cannot determine the solution of $\bfAhat
        \bfxhat = \bfbhat$.
	\end{proof}

    The situation that the \GMRES{} method terminates without finding the exact
    solution is often called a {\em breakdown} of \GMRES{}. The above proof
    leads to the following characterization of all initial guesses that lead to
    a breakdown of \GMRES{} at the first iteration.

    \begin{corollary}
        \label{corx0break}
        Let $\bfAhat\in\CC^{N\times N}$ and $\bfxhat,\bfbhat\in\CC^{N\times N}$
        such that $\bfAhat\bfxhat=\bfbhat$. Then the \GMRES{} method breaks down
        at the first iteration for all initial guesses
        \[
            \bfx_0\in\calX_0\DEF\big\{\bfxhat-\bfyhat ~\big|~
            \bfAhat\bfyhat\in\Ker(\bfAhat)\setminus\{\bf0\}\big\}.
        \]
    \end{corollary}

    We next have a closer look at condition~2 in Theorem~\ref{thmsinggmres}.  If
    we had $\Ker(\bfAhat) = \Ker(\bfAhat^\H)$, then $\Image(\bfAhat)^\perp =
    \Ker (\bfAhat^\H)$ would imply
	$$
        \left\{ \bf0 \right\} = \Image(\bfAhat)^\perp \cap \Image(\bfAhat) =
    	\Ker (\bfAhat^\H) \cap \Image(\bfAhat)=
    	\Ker (\bfAhat) \cap \Image(\bfAhat),
    $$
    so that condition~2 would hold. Thus condition~2 in
    Theorem~\ref{thmsinggmres} is fulfilled for any Hermitian matrix $\bfAhat$.
    For a general non-Hermitian matrix, however, it seems difficult to determine
    a deflated matrix with $\Ker(\bfAhat) = \Ker(\bfAhat^\H)$.
    However, for the
    deflated system \eqref{AhxbhGM} we can derive another condition that is
    equivalent with condition 2 (and hence condition 1) in
    Theorem~\ref{thmsinggmres}.

	\begin{corollary}\label{corgmresbreak}
        For the deflated system \eqref{AhxbhGM}, condition 2 in
        Theorem~\ref{thmsinggmres} is satisfied if and only if $\calU \cap
        (\bfA \calU)^\perp = \left\{ \bf0 \right\}$.  In particular, the latter
        condition is satisfied when $\calU$ is an exactly $\bfA$-invariant
        subspace, i.e., when $\bfA \calU = \calU$.
	\end{corollary}

	\begin{proof}
        Using the properties of the projection $\bfPA$ from Lemma~\ref{lemproj}
        and the fact that $\bfA$ is nonsingular, we obtain
		\begin{align*}
			\Ker(\bfAhat)&=	\Ker(\bfPA \bfA) = \bfA^{-1} \Ker(\bfPA) = \calU ,\\
			\Image(\bfAhat)&= \Image(\bfPA \bfA) = \Image(\bfPA) = (\bfA
            \calU)^\perp.
		\end{align*}
        If $\bfA
        \calU = \calU$, then $\calU \cap (\bfA \calU)^\perp = \left\{ \bf0
        \right\}$ holds trivially.
	\end{proof}

    For a nonsingular matrix condition 2 in Theorem~\ref{thmsinggmres} always
    holds trivially, and hence a breakdown of \GMRES{} can only occur if the
    method is applied to a linear algebraic system with a singular matrix (this
    fact has been known since the method's introduction in 1986~\cite{SaaS86}).
    Breakdowns have also been analyzed by de Sturler~\cite{Stu96} in the context
    of the \GCRO{} method (see the end of Section~\ref{secrminresbreak} below
    for further comments).  We want to point out that it is unlikely that a
    random initial guess lies in the subspace $\calX_0$ specified in
    Corollary~\ref{corx0break} for which the \GMRES{} method breaks down in the
    first step. However, a general $\calU$ may lead to a breakdown. To
    illustrate the problem of breakdowns in our context, we give an example that
    is adapted from \cite[Example~1.1]{BroW97}.

	\begin{example}\label{exbreakdown}
		{\rm Consider a linear algebraic system with
		$$\bfA = \Matrix{cc}{0 & 1\\1 & 0}, \quad \bfb = \ColVec{1\\0},$$
        so that the unique solution is given by the vector $[0,1]^\T$. Let the
        augmentation space be defined by $\bfU_1= [1,0]^\T$, then
		$$\bfPA = \Matrix{cc}{1 & 0\\0 &0},\quad
		\bfAhat = \bfPA \bfA = \Matrix{cc}{0 & 1 \\ 0 & 0}, \quad
		\bfbhat = \bfPA \bfb = \ColVec{1\\0}.$$
        If $\bfx_0$ is the zero vector, then $\bfrhat_0 = \bfbhat$ and $\bfAhat
        \bfrhat_0 = \bf0$, and thus \GMRES{} applied to the deflated system
        terminates at the very first iteration with the approximation $\bfx_0$.
        Since $\bfAhat \bfx_0 \neq \bfbhat$, this is a breakdown of \GMRES{}.
        Furthermore, applying the correction \eqref{xnP} to $\bfxhat_0 = \bfx_0$
        does not yield the solution of the original system $\bfA \bfx = \bfb$
        because
		\[
    		\bfQA \bfx_0 + \bfMA \bfA^\H \bfb = \bfMA \bfA^\H \bfb
            = \bfU_1 \bfU_1^\H \bfA^\H \bfb
	    	= \bf0 \neq \ColVec{0\\1}.
		\]
		}
	\end{example}

    Corollary~\ref{corgmresbreak} states that the \GMRES{} method applied to the
    deflated system~\eqref{AhxbhGM} cannot break down if $\calU$ is an
    $\bfA$-invariant subspace. The following example shows that care has also to
    be taken with approximate $\bfA$-invariant subspaces.
	\begin{example}\label{exbreakdown2}
        {\rm Let $\alpha>0$ be a small positive number. Then
        $\bfv\DEF[0,1,\alpha]^\T$ is an eigenvector of the matrix
        \[
            \bfA\DEF \Matrix{ccc}{0 & 1 & -\alpha^{-1} \\ 1 & 0 & \alpha^{-1} \\
            0 & 0 & 1}
        \]
        corresponding to the eigenvalue $1$. Instead of $\bfv$ we use the
        perturbed vector $\bfU_2\DEF[0,1,0]^\T$ as a basis for the deflation
        space $\calU=\Image(\bfU_2)$ and obtain
        \[
            \bfA\bfU_2=\Matrix{c}{1\\0\\0},\quad
            \bfPA=\Matrix{ccc}{0&0&0\\0&1&0\\0&0&1}, \quad
            \bfPA\bfA=\Matrix{ccc}{0&0&0\\1&0&\alpha^{-1}\\0&0&1}.
        \]
        For $\bfx,\bfb\in\CC^3$ with $\bfA\bfx=\bfb$ the \GMRES{} method then
        breaks down in the first step for all $\bfx_0\in\{\bfx+\beta [1,0,0]^\T
        ~|~ \beta\neq 0\}$. Note that $\|\bfU_2 - \bfv\|_2 = \alpha$ can be
        chosen arbitrarily small. A better measure for the quality of an
        approximate invariant subspace would be the largest principal angle
        between $\calU$ and $\bfA\calU$.
		}
	\end{example}

	\section{Hermitian matrices and variants of \MINRES}
	\label{secHermitianA}

    We will now apply the results presented in Sections~\ref{secdeflaug}
    and~\ref{secgeneralA} to the case where $\bfA$ is Hermitian, nonsingular,
    and possibly indefinite.  For a Hermitian matrix the \GMRES{} method
    considered in Section~\ref{secgeneralA} is mathematically equivalent to the
    \MINRES{} method, which is based on the Hermitian Lanczos algorithm, and
    thus uses efficient three-term recurrences.

	\subsection{The \RMINRES{} method}
	\label{secrminresbreak}

    This subsection discusses the ``recycling \MINRES{} method'', or briefly
    \RMINRES{} method, developed by Wang, de Sturler and Paulino~\cite{WanSP07}.
    This method fits into the framework of Section~\ref{secdeflaug}, and the
    results presented in Section~\ref{secgeneralA} apply.  Wang \ETAL were
    interested in solving sequences of linear algebraic systems
    that exhibit only small changes from one matrix in the sequence to the
    next one, and they
    suggested to reuse information from previous solves.  The \RMINRES{} method
    consists of two main parts that can basically be analyzed separately: an
    augmented and deflated \MINRES{} solver which is based on \GCRO{} and an
    extraction procedure for the augmentation and deflation data.  In the second
    part Wang \ETAL determined harmonic Ritz vectors that correspond to harmonic
    Ritz values close to zero, and used these approximate eigenspaces for
    augmenting the Krylov subspace.  Here, we omit the extraction of the
    augmentation and deflation space and concentrate on the method for solving
    the systems. We refer to this as the {\em solver part of the \RMINRES{}
    method}.
    We point out that the extracted spaces can be arbitrary if
    there are no restrictions on the changes of the matrices in the sequence of
    linear algebraic systems.  However, the \RMINRES{} method has been presented
    in \cite{WanSP07} with an application in topology optimization where the
    extracted approximated eigenvectors of one matrix are still good
    approximations to eigenvectors of the next matrix. Furthermore, we will not
    address the preconditioning technique outlined in \cite{WanSP07} and assume
    that the given linear algebraic system is already in the preconditioned form.

    As in Section~\ref{secgeneralA}, we set $\bfB= \bfA$ and consider first the
    resulting deflated system of the form \eqref{AhxbhGM},
	\begin{equation*}
		\bfAhat \bfxhat =\bfbhat,\quad\mbox{where}\quad
		\bfAhat \ASS \bfPA \bfA,\quad \bfbhat \ASS \bfPA \bfb.
	\end{equation*}
    If we apply \MINRES{} to this linear algebraic system with an initial guess
    $\bfx_0$ and the corresponding initial residual
    $\bfrhat_0=\bfbhat-\bfAhat\bfx_0=\bfPA (\bfb-\bfA\bfx_0)$, then the iterate
    $\bfxhat_n$ and the residual $\bfrhat_n$ are characterized by the two
    conditions
	\begin{align}\label{minres}
		\bfxhat_n \in \bfx_0 + \calK_n(\bfAhat, \bfrhat_0), \quad\mbox{and}\quad
		\bfrhat_n = \bfbhat - \bfAhat \bfxhat_n \perp \bfAhat \calK_n(\bfAhat, \bfrhat_0).
	\end{align}
    This is essentially the approach of Kilmer and de
    Sturler~\cite[Section~2]{KilS06}.  Olshanskii and Simoncini~\cite{OlsS10}
    recently used a different approach where the \MINRES{} method is applied to
    the deflated system $\bfPI\bfA\bfxhat=\bfPI\bfb$ with the special initial
    guess $\bfx_0=\bfU(\bfU^\H\bfA\bfU)^{-1}\bfU^\H\bfb$. We note that the
    presentation in \cite{OlsS10} is slightly different but the above can be
    seen with minor algebraic modifications to the relations in and preceding
    Proposition~3.1 in \cite{OlsS10}.

    An attentive reader has certainly noticed that the deflated matrix $\bfAhat
    = \bfPA \bfA = \bfA - \bfA \bfMA \bfA^2$ is in general {\em not} Hermitian,
    even when $\bfA$ is Hermitian. However, as pointed out in~\cite[footnotes on
    p.~2153 and p.~2446, respectively]{KilS06,WanSP07}, a straightforward
    computation shows that
	\begin{equation}
		\calK_n (\bfPA \bfA, \bfPA \bfv) = \calK_n (\bfPA \bfA \bfPA, \bfPA \bfv) \label{KPAP}
	\end{equation}
    holds for every vector $\bfv\in\CC^N$ because $\bfPA$ is a projection. The
    matrix $\bfPA \bfA \bfPA$ is obviously Hermitian (since $\bfA$ and $\bfPA$
    are Hermitian), and hence the Krylov subspaces we work with are also
    generated by a Hermitian matrix.  It is therefore possible to implement a
    \MINRES-like method for the deflated system, which is based on three-term
    recurrences and which is characterized by the conditions \eqref{minres}.  As
    presented in Section~\ref{secgeneralA}, these conditions combined with the
    correction step \eqref{xnP"} are equivalent to the explicit use of
    augmentation, i.e., conditions \eqref{xnKnU}--\eqref{rnAKnAU}.

    The latter conditions are the basis of the solver part of the \RMINRES{}
    method by Wang, de Sturler and Paulino in~\cite[Section~3]{WanSP07}.  We
    summarize the above and give two mathematically equivalent characterizations
    of the \RMINRES{} solver applied to the system $\bfA \bfx = \bfb$ with an
    initial guess $\bfx_0$:
	\begin{enumerate}
        \item The original approach used in~\cite{WanSP07} incorporates explicit
            augmentation, which means to construct iterates $\bfx_n$ satisfying
            the two conditions
			\begin{align}
				\begin{aligned}
					\bfx_n &\in \bfx_0 + \calK_n(\bfPA \bfA, \bfPA \bfr_0) + \calU, \\
					\bfr_n &= \bfb - \bfA \bfx_n
                       \perp \bfA \calK_n(\bfPA \bfA, \bfPA \bfr_0) + \bfA \calU.
				\end{aligned}
				\label{rminresDS}
			\end{align}

        \item A mathematically equivalent approach is to apply \MINRES{} to the
            deflated system
			\begin{align}
				\bfPA \bfA \bfxhat = \bfPA \bfb
				\label{rminresMR}
			\end{align}
			and correct the resulting iterates $\bfxhat_n$ according to
			$\bfx_n = \bfQA \bfxhat_n + \bfMA \bfA \bfb$.
	\end{enumerate}

    Note that on an algorithmic level the second approach exhibits lower
    computational cost since the correction in the space $\calU$ is only carried
    out once at the end, while the \RMINRES{} solver requires one update per
    iteration.

    Since the solver part of \RMINRES{} is mathematically equivalent to
    \MINRES{} (and \GMRES{}) applied to the deflated system,
    Corollary~\ref{corgmresbreak} also applies to \RMINRES{}.  In particular,
    the method can break down for specific initial guesses if (and only if)
    $\calU \cap (\bfA \calU)^\perp \neq \left\{ \bf0 \right\}$.  Breakdowns
    cannot occur if $\calU$ is an exact $\bfA$-invariant subspace, but this is
    an unrealistic assumption in practical applications.  Note that the matrix
    $\bfA$ in Example~\ref{exbreakdown} is Hermitian, thus it also serves as an
    example for a breakdown of the \RMINRES{} solver.
    That the \RMINRES{}
    method can break down may already be guessed from the fact that this method
    is based on the \GCRO{} method and thus potentially suffers from the
    breakdown conditions for \GCRO{} derived in~\cite{Stu96}.
    However, the
    possibility of breakdowns has not been mentioned
    in~\cite{WanSP07}, and in the example of a \GCRO{} breakdown given in \cite{Stu96}
    the matrix $\bfA$ is not Hermitian. Hence this example cannot be used in the
    context of the \RMINRES{} method, which is intended for Hermitian matrices.

    In the next subsection we show how to suitably modify the \RMINRES{}
    approach to avoid breakdowns.

    \subsection{Avoiding breakdowns in deflated \MINRES{}}
	\label{secdeflminres}

    We have seen in Section~\ref{secgeneralA} that if $\Ker(\bfAhat) =
    \Ker(\bfAhat^\H)$, then condition 1 in Theorem~\ref{thmsinggmres} is
    satisfied. Consequently, if we can determine a {\em Hermitian deflated
    matrix} $\bfAhat$ and a corresponding consistent deflated system, \MINRES{}
    applied to this system cannot break down for any initial guess.

    Using the projections $\bfPA$ and $\bfQA$ from \eqref{QPP} we decompose the
    solution $\bfx$ of $\bfA \bfx = \bfb$ as
	\begin{align}
		\bfx &= \bfPA \bfx + ( \bfI - \bfPA ) \bfx = \bfPA \bfx + \bfA \bfMA \bfA \bfx
		= \bfPA \bfx + \bfA \bfMA \bfb, \label{xP} \\
		\bfx &= \bfQA \bfx + ( \bfI - \bfQA ) \bfx = \bfQA \bfx + \bfMA \bfA^2 \bfx
		= \bfQA \bfx + \bfMA \bfA \bfb. \label{xPt}
	\end{align}
    Using \eqref{xPt}, the system $\bfA \bfx = \bfb$ becomes $\bfA ( \bfQA \bfx
    + \bfMA \bfA \bfb ) =\bfb$.  With the definition of $\bfPA$ and  $\bfA \bfQA
    = \bfPA \bfA$ (cf.~Lemma~\ref{lemproj}) we see that this is equivalent to
	\[
        \bfPA \bfA \bfx = \bfPA \bfb.
    \]
    We now substitute for $\bfx$ from \eqref{xP} and obtain $\bfPA \bfA ( \bfPA
    \bfx + \bfA \bfMA \bfb) = \bfPA \bfb$ which is equivalent to
	\begin{align}
		\bfPA \bfA \bfPA \bfx = \bfPA \bfQA^\H \bfb.
		\label{PAPxPPtb}
	\end{align}

    We can show the following result for the \MINRES{} method applied to this
    symmetric system.

	\begin{theorem}
		\label{thmdeflminr}
        For each initial guess $\bfx_0 \in \CC^N$ the \MINRES{} method applied
        to the system \eqref{PAPxPPtb} yields $($in exact arithmetic$)$ a
        well-defined iterate $\bfxb_n$ at every step $n\geq 1$ until it
        terminates with a solution. Moreover, the sequence of iterates
		\begin{align}
			\bfx_n \ASS \bfQA \left( \bfPA \bfxb_n + \bfA \bfMA \bfb \right) + \bfMA \bfA \bfb
			\label{xnCorr}
		\end{align}
        is well defined. It terminates $($in exact arithmetic$)$ with the exact
        solution $\bfx$ of the original linear system $\bfA \bfx = \bfb$, and
        its residuals are given by $\bfr_n=\bfb - \bfA \bfx_n = \bfPA \bfQA^\H
        \bfb - \bfPA \bfA \bfPA \bfxb_n$.
	\end{theorem}

	\begin{proof}
        The first part follows from the fact that the system \eqref{PAPxPPtb}
        is a consistent system with a Hermitian matrix $\bfPA \bfA \bfPA$, so
        that we can apply Theorem~\ref{thmsinggmres}.
        It remains to show the second part.	The $n$-th residual of the original
        system $\bfA \bfx = \bfb$ is given by
		\begin{align*}
			\bfr_n &= \bfb - \bfA \bfx_n
			= \bfb - \bfA \left( \bfQA \left( \bfPA \bfxb_n + \bfA \bfMA \bfb \right) + \bfMA \bfA \bfb \right) \\
			&= \bfb - \bfA \bfQA \left( \bfPA \bfxb_n + \bfA \bfMA \bfb \right) - \bfA \bfMA \bfA \bfb \\
			&= \left( \bfI - \bfA \bfMA \bfA \right) \bfb - \bfPA \bfA \left( \bfPA \bfxb_n + \bfA \bfMA \bfb \right) \\
			&= \bfPA \bfb - \bfPA \bfA \bfPA \bfxb_n - \bfPA \bfA^2 \bfMA \bfb \\
			&= \bfPA \left( \bfI - \bfA^2 \bfMA \right) \bfb - \bfPA \bfA \bfPA \bfxb_n \\
			&= \bfPA \bfQA^\H \bfb - \bfPA \bfA \bfPA \bfxb_n.
		\end{align*}
        We see that $\bfr_n$ is equal to the $n$-th \MINRES{} residual for the
        system \eqref{PAPxPPtb}.  In particular, this implies that the exact
        solution of \eqref{Axb} is given by \eqref{xnCorr} once an exact
        solution $\bfxb_n$ of \eqref{PAPxPPtb} has been determined by \MINRES{}.
	\end{proof}

    When \MINRES{} is applied to the deflated system~\eqref{PAPxPPtb}, the
    Hermitian iteration matrix $\bfPA\bfA\bfPA$ can again be replaced by the
    non-Hermitian matrix $\bfPA\bfA$ (cf.~Section~\ref{secrminresbreak}).

    The following theorem shows that a modification of the initial guess
    suffices to make the solver part of the \RMINRES{} method mathematically
    equivalent to \MINRES{} applied to the system \eqref{PAPxPPtb}.

	\begin{theorem}
		\label{thmrminresx0}
		We consider the following two approaches:
		\begin{enumerate}
            \item The solver part of the \RMINRES{} method applied to $\bfA \bfx =
                \bfb$ with the initial guess $\bfxhat_0 \ASS \bfPA \bfx_0 + \bfA
                \bfMA \bfb$ and resulting iterates $\bfx_n$ and residuals
                $\bfr_n = \bfb - \bfA \bfx_n$.
            \item \label{rmindefl} The \MINRES{} method applied to
                \eqref{PAPxPPtb} with the initial guess $\bfx_0$ and resulting
                iterates $\bfxb_n$ and residuals $\bfrb_n \ASS \bfPA \bfQA^\H
                \bfb - \bfPA \bfA \bfPA \bfxb_n$.
		\end{enumerate}
        Both approaches are equivalent in the sense that $\bfx_n = \bfQA ( \bfPA
        \bfxb_n + \bfA \bfMA \bfb) + \bfMA \bfA \bfb$ and $\bfr_n = \bfrb_n$.
	\end{theorem}

	\begin{proof}
        Let us start with the \MINRES{} method applied to \eqref{PAPxPPtb},
        which constructs iterates $\bfxb_n = \bfx_0 + \bfV_n \bfy_n$, where
        $\bfV_n \in \CC^{N \times n}$ is of full rank $n$ such that
        $\Image(\bfV_n)=\calK_n(\bfPA \bfA \bfPA ,\bfP \bfQ^\H \bfr_0)$.  Then
        $\bfPA \bfV_n =\bfV_n$ and the corrected iterates are
		\begin{align}
			\bfx_n & = \bfQA ( \bfPA (\bfx_0 + \bfV_n \bfy_n) + \bfA \bfMA \bfb) + \bfMA \bfA \bfb
			= \bfQA ( \bfxhat_0 + \bfV_n \bfy_n ) + \bfMA \bfA \bfb \notag \\
			&= \bfQA \bfxhat_n + \bfMA \bfA \bfb, \label{MRcorr}
		\end{align}
        with $\bfxhat_n \DEF \bfxhat_0 + \bfV_n \bfy_n$. For $n>0$ the $n$-th
        residual of $\bfxb_n$ with respect to the system \eqref{PAPxPPtb} is
        \begin{align*}
			\bfrb_n &= \bfPA \bfQA^\H \bfb - \bfPA \bfA \bfPA \bfxb_n
			= \bfPA ( \bfQA^\H \bfb - \bfA \bfPA \bfx_0 - \bfA \bfV_n \bfy_n) \\
			&= \bfPA( \bfb - \bfA (\bfPA \bfx_0 + \bfA \bfMA \bfb + \bfV_n \bfy_n))
			= \bfPA \bfb - \bfPA \bfA \bfxhat_n \FED \bfrhat_n.
		\end{align*}
        This is the residual of $\bfxhat_n$ with respect to the system
        \eqref{rminresMR}. We also have
		\begin{align*}
			\bfrhat_0 = \bfPA \bfb -\bfPA \bfA \bfxhat_0
            = \bfPA \bfQA^\H \bfb - \bfPA \bfA \bfPA \bfx_0 = \bfrb_0,
		\end{align*}
        and thus the starting vectors of the Krylov subspace for both methods
        are equal. Because of \eqref{KPAP} also the Krylov subspaces are equal.
        From the definition of the Krylov subspaces we immediately obtain
		\begin{align*}
			\bfrb_n \perp \bfPA \bfA \bfPA \calK_n(\bfPA \bfA \bfPA ,\bfrb_0)
			\quad \LLRA \quad
			\bfrhat_n \perp \bfPA \bfA \calK_n(\bfPA \bfA, \bfrhat_0).
		\end{align*}
        We can now see that the iterates $\bfxhat_n$ are the iterates of
        \MINRES{} applied to \eqref{rminresMR} with the initial guess $\bfxhat_0$.
        Along with the correction \eqref{MRcorr} this was shown to be equivalent
        to the \RMINRES{} solver applied to $\bfA \bfx = \bfb$ with the initial guess
        $\bfxhat_0$ (cf.~Section~\ref{secrminresbreak}).
	\end{proof}

    This means that (in exact arithmetic) breakdowns in the solver part of the
    \RMINRES{} method can be prevented by either adapting the right-hand side to
    $\bfQ^\H\bfb$ and correcting the approximate solution at the end according
    to Theorem~\ref{thmdeflminr} or by choosing the adapted initial guess
    $\bfxhat_0$ defined in Theorem~\ref{thmrminresx0}. Both choices do not
    increase the computational cost significantly since these computations only
    need to be carried out once. A similar special initial guess has also been
    used in~\cite{TanNVE09} to obtain a robust deflation-based preconditioner
    for the \CG{} method; compare the A-DEF2 method in~\cite[Table~2]{TanNVE09}.

    \subsection{Numerical experiments}
    \label{secnumexp}

    In this subsection, we will show the numerical behavior of selected Krylov
    subspace methods discussed above.  Detailed numerical experiments with the
    deflated \CG{} method (cf.~Section~\ref{secHpdA}) and equivalent approaches
    have been presented in~\cite{TanNVE09}.  Here, we will focus on the solver
    part of the \RMINRES{} method and the deflated \MINRES{} method in order to
    numerically illustrate the phenomenon of breakdowns that have only been
    described theoretically so far (cf.~Sections~\ref{secrminresbreak} and
    \ref{secdeflminres}).  Both methods are implemented in MATLAB with
    three-term Lanczos recurrences and Givens rotations for solving the least
    squares problem. All residuals have been computed explicitly in each
    iteration.

	\begin{example} \label{ex:evals}
		\rm
        In this example we use a matrix $\bfA=\bfW^\H \bfD \bfW \in \RR^{2m
        \times 2m}$, $m=50$, where $\bfD=\diag(\lambda_1,\dots,\lambda_{2m})$
        with $\lambda_j=\sqrt{j}$, $\lambda_{m+j}=-\sqrt{j}$ for $j=1,\dots,m$
        and $\bfW = [\bfw_1, \dots,\bfw_{2m}]$ is a randomly generated
        orthogonal matrix. We consider a matrix $\bfU=[u_1,\dots,u_k]$ whose
        columns are pairwise orthogonal eigenvectors of $\bfA$, i.e. $\bfA \bfU
        = \bfU \bfD_\bfU$ and $\bfU^\H \bfU = \bfI_k$ with a diagonal matrix
        $\bfD_\bfU=\diag(\lambda_{j_1},\dots,\lambda_{j_k})$ for
        $0<j_1<\dots<j_k<2m$. This means that $\calU=\Image(\bfU)$ is an exact
        $\bfA$-invariant subspace.  Then a straightforward computation reveals
        that $\bfPA= \bfQA = \bfI - \bfU \bfU^\H$, which is obviously Hermitian,
        and
		\[
			\bfPA \bfA \bfPA = \bfPA \bfA \bfQA = \bfPA^2 \bfA = \bfPA \bfA, \qq
			\bfPA \bfQA^\H = \bfPA^2 =\bfPA.
		\]
        By comparing the correction steps of \RMINRES{} and deflated \MINRES{}
        (cf.~Sections~\ref{secrminresbreak} and \ref{secdeflminres}) and using
        $\bfPA \bfA \bfMA = \bf0$, we can see that both methods are
        mathematically equivalent if $\calU$ is an exact invariant subspace.

		\begin{figure}
			\begin{center}
				\includegraphics[width=0.8\textwidth]{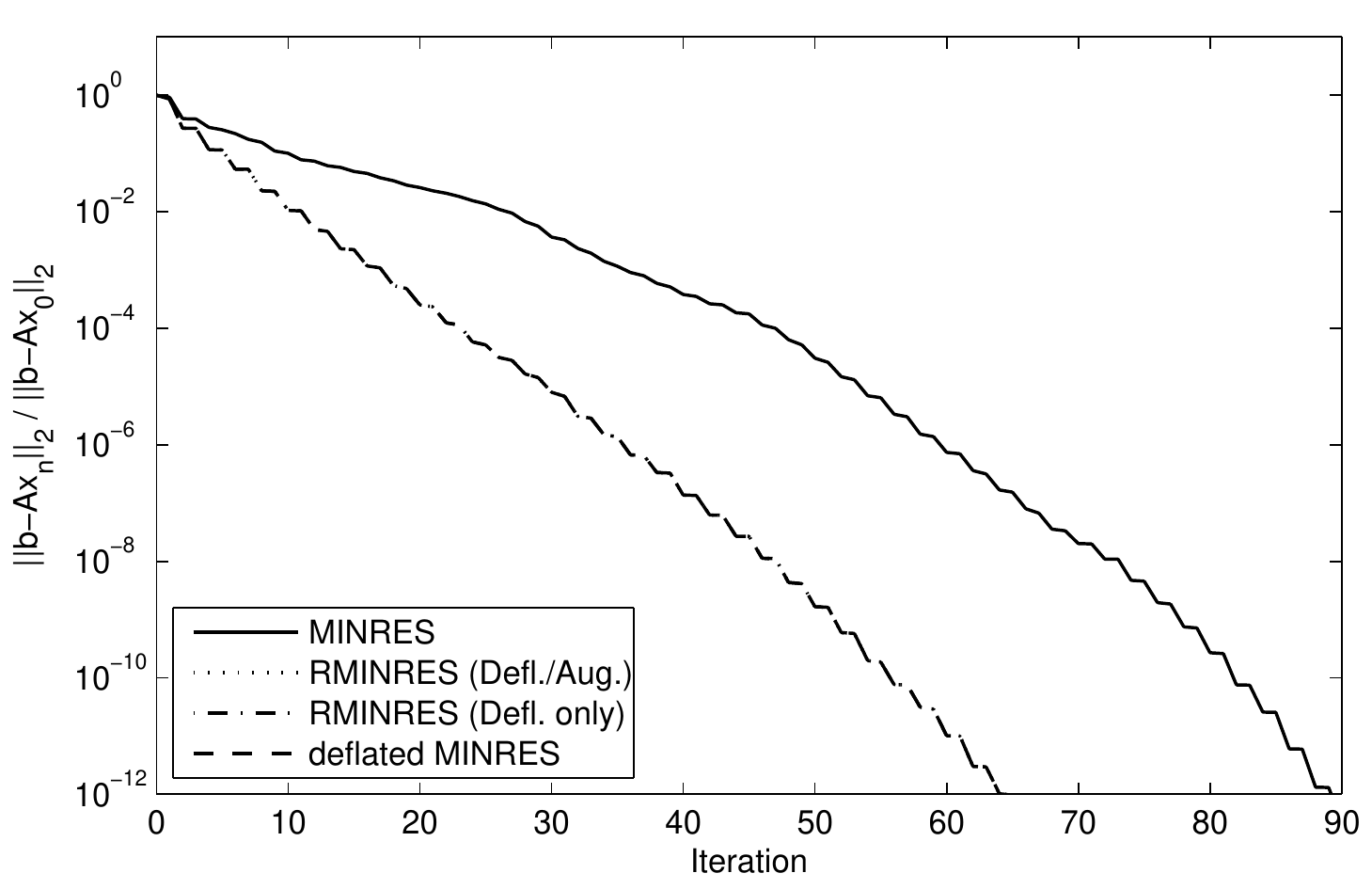}
			\end{center}
            \caption{Convergence history for Example \ref{ex:evals}.  The
            convergence curves of both \RMINRES{} solver implementations and the
            deflated \MINRES{} method coincide.  }
			\label{fig:evals}
		\end{figure}

        We solve the system $\bfA \bfx=\bfb$ with a random right-hand side
        $\bfb$ and the initial guess $\bfx_0=\bf0$. In Figure~\ref{fig:evals} we
        show the relative residual norms of the solvers
		\begin{itemize}
			\item \MINRES{} (solid line),
            \item \RMINRES{} with explicit augmentation and deflation (dotted
                line) according to Wang \ETAL~\cite{WanSP07}; 
                cf.~\eqref{rminresDS},
            \item \RMINRES{} with deflation only (dash-dotted line), i.e., the
                residual norms of \MINRES{} applied to the system $\bfPA \bfA
                \bfx = \bfPA \bfb$; cf.~\eqref{rminresMR},
            \item deflated \MINRES{} (dashed line), i.e., the residual norms of
                \MINRES{} applied to the system $\bfPA \bfA \bfPA \bfx = \bfPA
                \bfQA^\H \bfb$; cf.~Section~\ref{secdeflminres}.
		\end{itemize}
        For the last three methods we used the matrix
        $\bfU=[\bfw_1,\dots,\bfw_5,\bfw_{51},\dots,\bfw_{55}]$ which contains
        the eigenvectors associated with the $10$ eigenvalues of $\bfA$ of
        smallest absolute value.  Thus the deflation space $\calU$ has dimension
        $10$. We have shown above that
        the two implementations of \RMINRES{} and the deflated
        \MINRES{} method are mathematically equivalent, and in this example the
        three convergence curves corresponding to these methods indeed coincide;
        see Figure~\ref{fig:evals}.
	\end{example}

	\begin{example} \label{ex:break}
		\rm
        We now investigate breakdowns and near-breakdowns of the
        \RMINRES{} method using a set of artificially constructed examples. Of
        course, the occurrence of an exact breakdown as in the following
        examples will be rare in practical applications.

        For our construction
        we use the same matrix $\bfA$ as in
        Example~\ref{ex:evals} and we construct a subspace $\calU$ for which
        $\calU \cap (\bfA \calU)^\perp \neq \{\bf0\}$.  Thus, the condition that
        guarantees a breakdown-free \RMINRES{} computation is violated; 
        cf.~Section~\ref{secrminresbreak}.
        To construct the subspace $\calU$ we choose an integer $k$, $0<k<m$, and
        we define $\bfW_1=[\bfw_{i_1},\dots,\bfw_{i_k}]$ and
        $\bfW_2=[\bfw_{m+i_1},\dots,\bfw_{m+i_k}]$ for indices $0 < i_1
        <\dots<i_k <m$.  With
        $\bfD_\bfU=\diag(\lambda_{i_1},\cdots,\lambda_{i_k})$ we obtain $\bfA
        \bfW_1 = \bfW_1 \bfD_\bfU$ and $\bfA \bfW_2 = - \bfW_2 \bfD_\bfU$
        because of the symmetry of the spectrum of $\bfA$.  We now choose the
        matrix $\bfU = \bfW_1 + \bfW_2$. Applying $\bfA$ yields $\bfA \bfU =
        (\bfW_1 - \bfW_2) \bfD_\bfU$ and using the fact that $\bfW$ is unitary
        shows that $\bfU^\H \bfA \bfU = \bf0$, or equivalently $\calU \subset
        (\bfA \calU)^\perp$.  The proof of Theorem~\ref{thmsinggmres} gives us a
        way to construct an initial guess which leads to an immediate breakdown
        of \RMINRES{}.  For an arbitrary $\bf0 \neq \bfu \in \calU$ we choose
        $\bfx_0 = \bfA^{-1} (\bfb - \bfu)$.  Because of $\calU \perp \bfA \calU$
        we have $\bfPA \bfu = \bfu$ and the initial residual of \RMINRES{} is
        $\bfr_0 = \bfPA \bfb - \bfPA \bfA \bfx_0 = \bfu$. The breakdown then
        occurs in the first iteration because $\bfPA \bfA \bfr_0 = \bfPA \bfA
        \bfu = \bf0$ since $\bfA \bfu \in \bfA \calU = \Ker(\bfPA)$.
        For these constructed initial guesses the \RMINRES{} method indeed
        breaks down immediately in numerical experiments, whereas the deflated
        \MINRES{} method finds the solution after one step.
        There is no need to plot these results.

		\begin{figure}
			\begin{center}
				\subfigure[Unperturbed deflation space $\bfU^{(1)}$]{
                    \includegraphics[width=0.8\textwidth]{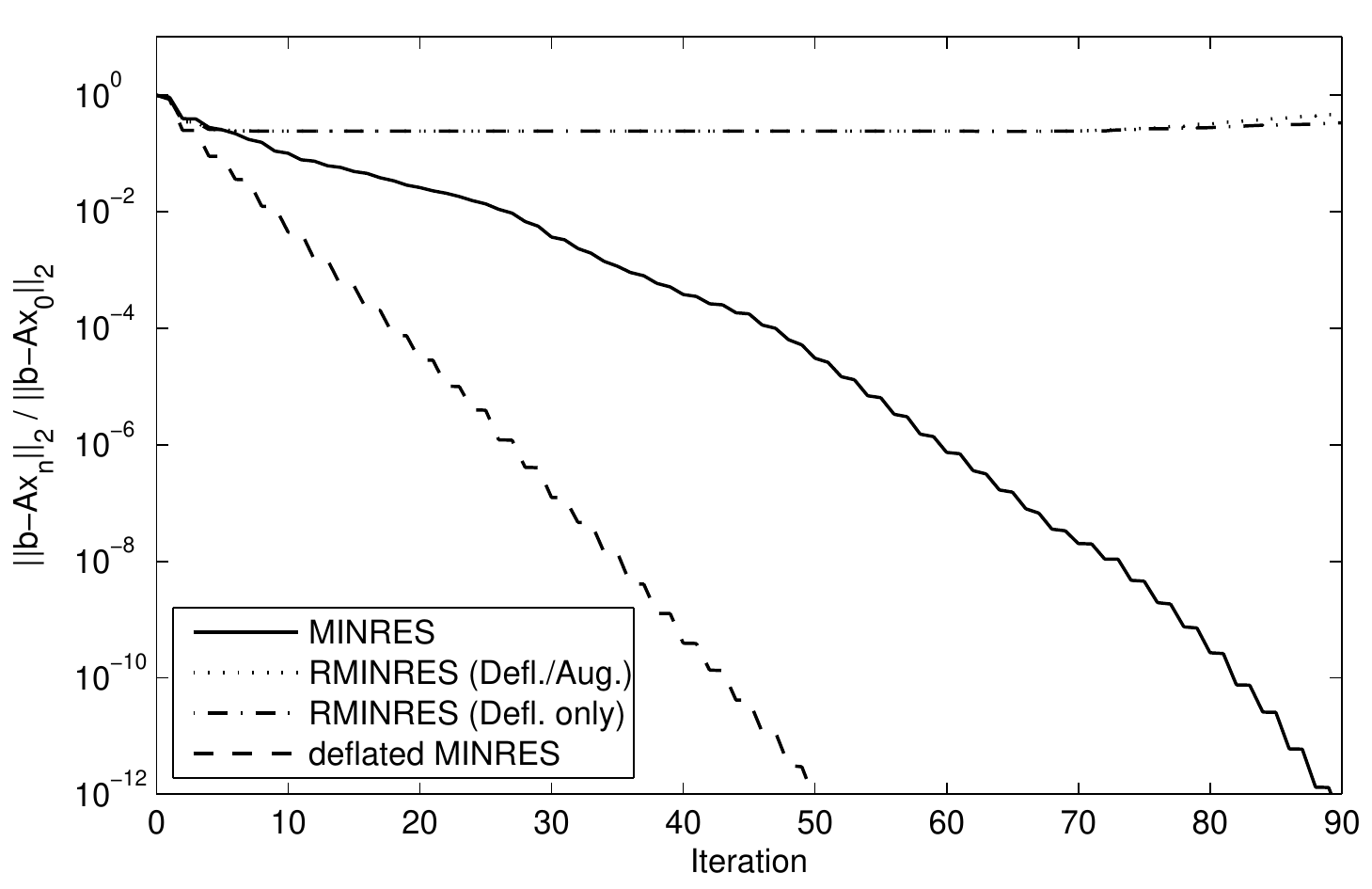}
                    \label{fig:break1}
				}
				\subfigure[Perturbed deflation space $\bfU^{(2)}= \bfU^{(1)} + \bfE$]{
                    \includegraphics[width=0.8\textwidth]{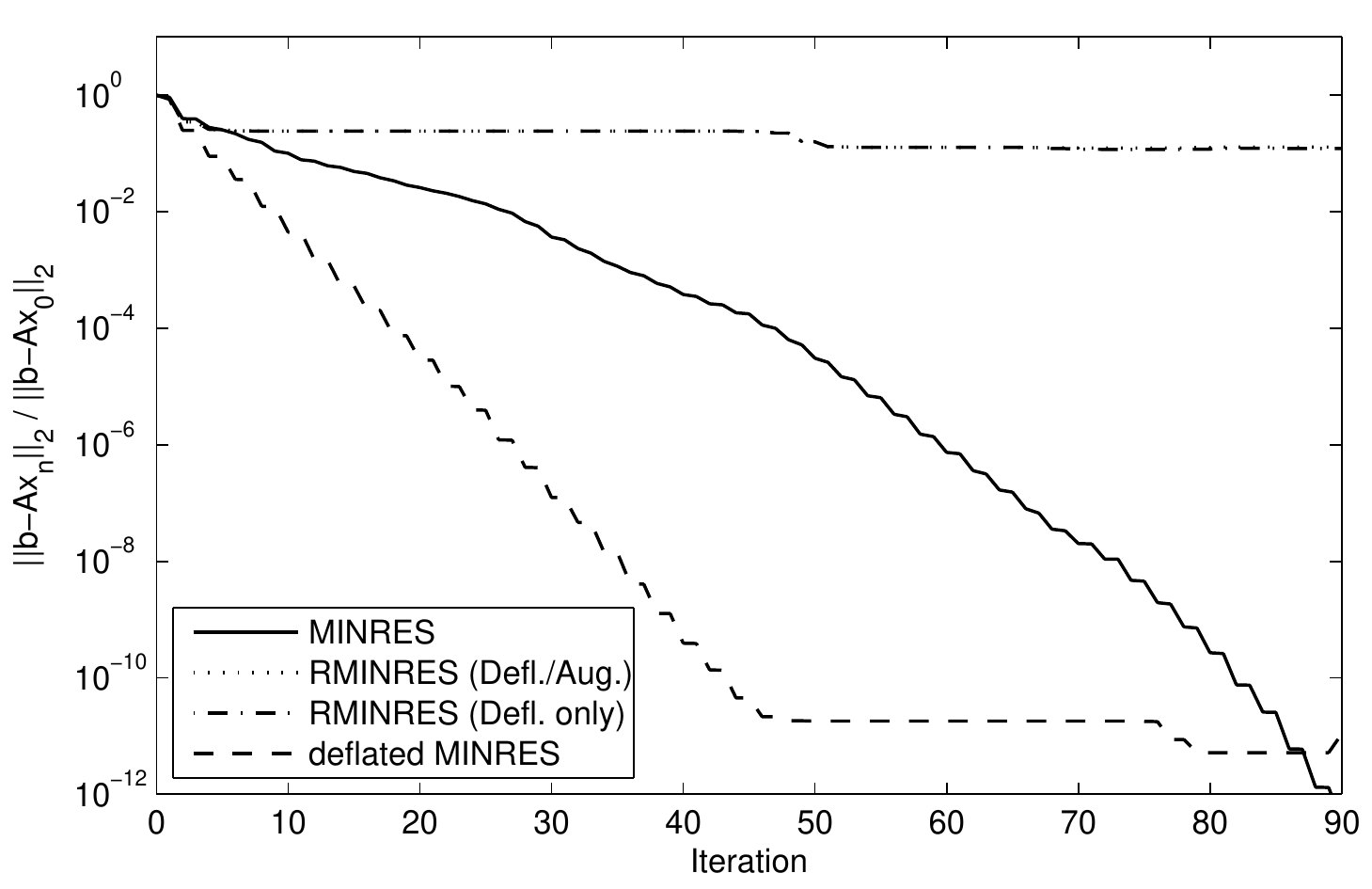}
                    \label{fig:break2}
				}
			\end{center}
            \caption{Convergence history for Example \ref{ex:break}.  The
            convergence curves of both \RMINRES{} solver implementations
            coincide.  }
			\label{fig:break}
		\end{figure}

        Of greater
        interest are situations with perturbed data.  Interestingly,
        randomly perturbed initial guesses lead to a breakdown of
        \RMINRES{} with the previously constructed deflation space as well.
        In Figure~\ref{fig:break1} we show the relative residual norms of the
        solvers listed above applied to the same $\bfA$ and $\bfb$ as in the
        previous example and with the matrix $\bfU^{(1)} =
        [\bfw_1+\bfw_{51},\dots,\bfw_{10}+\bfw_{60}]$.

        Furthermore, breakdowns also occur when we perturb the deflation space.
        Figure~\ref{fig:break2} shows the results for a perturbed matrix
        $\bfU^{(2)}= \bfU^{(1)} + \bfE$ with a random $\bfE \in \CC^{100 \times
        10}$ and $||\bfE||_2=10^{-10}$. The used initial guess is the same perturbed
        initial guess as in the experiment conducted for
        Figure~\ref{fig:break1}.

        Note that both \RMINRES{} implementations suffer from a breakdown after
        a few steps with both matrices $\bfU^{(1)}$ and $\bfU^{(2)}$. With the
        unperturbed matrix $\bfU^{(1)}$ the deflated \MINRES{} method converges
        to the solution with a relative residual smaller than $10^{-12}$, while in
        the case of the perturbed matrix $\bfU^{(2)}$ the method stagnates with
        a relative residual of order $10^{-11}$.  This stagnation of deflated
        \MINRES{} seems to be related to an unfavorable spectrum of $\bfPA \bfA
        \bfPA$ for these specifically constructed and perturbed matrices like
        $\bfU^{(2)}$. It is unlikely that the stagnation is caused by roundoff
        errors because the stagnation also occurs (up to iteration 100) when
        full recurrences (\GMRES{}) are used instead of short recurrences
        (\MINRES{}).  Perturbing the matrix $\bfU$ from Example~\ref{ex:evals}
        whose columns are exact eigenvectors of $\bfA$ does not cause
        stagnation.  This behavior is still subject to further research.

        Note that the construction of $\calU=\Image(\bfU)$ in
        Example~\ref{ex:break} is such that $\bfA\calU\perp\calU$ which cannot
        be achieved with a (nearly) $\bfA$-invariant subspace if $\bfA$ is
        Hermitian. In~\cite{WanSP07} an approximation to an invariant subspace
        of a previous matrix in a sequence of linear algebraic systems is used.
        In this situation care has to be taken that the extracted space is still
        a good approximation to an invariant subspace of the current matrix.
        However, in the experiments of \cite[Section 7]{WanSP07} this seems to
        be fulfilled since stagnation has not been observed.
	\end{example}

	\section{Conclusions}

    In this paper we first analyzed theoretically the link between basic
    theoretical properties of deflated and augmented Krylov subspace methods
    whose residuals satisfy a Galerkin condition, including the minimum residual
    methods whose inclusion into the class of Galerkin methods requires a
    replacement of the standard inner product.  We proved that augmentation can
    be achieved without explicitly augmenting the Krylov subspace, but instead
    projecting the residuals appropriately and using a correction formula for
    the approximate solutions.  We discussed this result in detail for the \CG{}
    method and \GMRES{}/\MINRES{} methods, the main representatives of our
    class.  It turned out that for these methods some of our results had been
    mentioned before in the literature.

    The projections which arise from the augmentation can also be used to obtain
    a deflated system. We have seen that a left-multiplication of the original
    system with the corresponding projection yields a deflated system for which
    the \CG{} method and \GMRES{}/\MINRES{} methods implicitly achieve
    augmentation. We proved that for nonsingular Hermitian matrices the
    \MINRES{} method for the deflated system is equivalent to the solver part of
    the \RMINRES{} method introduced in~\cite{WanSP07}. While \CG{} never breaks
    down, \GMRES{}, \MINRES{} and thus \RMINRES{} may suffer from breakdowns
    when used with the deflated systems. We stated necessary and sufficient
    conditions to characterize breakdowns of these minimal residual methods. For
    Hermitian matrices, we introduced the deflated \MINRES{} method which also
    uses a Hermitian deflated matrix and proved that it cannot break down.
    These results were illustrated numerically.

    Our framework covers methods based on a specific type of Galerkin condition;
    see \eqref{xn}-\eqref{rn}. It does not include methods based on other
    conditions, in particular those that in practical methods are realized using
    the non-Hermitian Lanczos algorithms. Examples for such methods are
    \BICG{}~\cite{Fle76} and its variants including \CGS{}~\cite{Son89},
    \BICGSTAB{}~\cite{Vor92}, and \IDRS{}~\cite{SonG08}. Extending our framework
    to such methods remains a subject of further work.

    Moreover, in this paper we did not discuss or recommend practical choices of
    deflation and/or augmentation spaces. Finding spaces that lead to an
    improved convergence behavior of the deflated or augmented method is a
    highly challenging task that should be attacked with a specific application
    in mind. Similar to preconditioning, there exists no single-best strategy
    for choosing deflation or augmentation spaces in practice. Often one
    deflates (approximations of) eigenvectors corresponding to the smallest
    eigenvalues of the given matrix. For symmetric or Hermitian positive
    definite matrices, this
    strategy can be shown to reduce the ``effective condition number'', which in
    turn leads to improved convergence bounds, and actually faster convergence
    of the iterative solver; see e.g. \cite{VuiNT06}. For non-symmetric or
    non-Hermitian matrices, however, the question of effective choices of
    deflation or augmentation spaces is largely open.

    \section*{Acknowledgements}

    The authors wish to thank the anonymous referees for their comments which
    helped to improve the presentation.

	\bibliographystyle{siam}
	\bibliography{deflaug}

\end{document}